\documentclass[11pt,a4paper,leqno]{amsart}
\usepackage{amsfonts}
\usepackage{amsthm}
\usepackage{amsmath}
\usepackage{amscd}
\usepackage[latin2]{inputenc}
\usepackage{t1enc}
\usepackage[mathscr]{eucal}
\usepackage{indentfirst}
\usepackage{graphicx}
\usepackage{graphics}
\usepackage{pict2e}
\usepackage{epic}
\usepackage{stmaryrd}
\usepackage{authblk}
\usepackage{subcaption}
\usepackage[symbol]{footmisc}
\numberwithin{equation}{section}
\usepackage[margin=2.9cm]{geometry}
\usepackage{epstopdf}
\RequirePackage{doi}
\usepackage{hyperref}
\allowdisplaybreaks
\usepackage{cite}

\usepackage{verbatim,amssymb,amsfonts}
\usepackage{mathrsfs}
\usepackage{mathtools}
\usepackage{tikz-cd}
\usepackage{indentfirst}
\usepackage{slashed}
\usepackage{thmtools}
\usepackage{setspace}
\usepackage{enumerate}
\theoremstyle{plain}
\newtheorem{theorem}{Theorem}[section]
\newtheorem{lemma}[theorem]{Lemma}

\newtheorem{proposition}[theorem]{Proposition}

\newtheorem{assumption}{Assumption}

 \theoremstyle{definition}
\newtheorem{definition}[theorem]{Definition}
\newtheorem{remark}[theorem]{Remark}

\newtheorem{manualtheoreminner}{Theorem}
\newenvironment{manualtheorem}[1]{%
  \renewcommand\themanualtheoreminner{#1}%
  \manualtheoreminner
}{\endmanualtheoreminner}

\newtheorem{manuallem}{Lemma}
\newenvironment{manuallemma}[1]{%
  \renewcommand\themanuallem{#1}%
  \manuallem
}{\endmanuallem}

\newcommand{\im}{\operatorname{im}}

\DeclarePairedDelimiterX{\inp}[2]{\langle}{\rangle}{#1, #2}

\newcommand{\cB}{{\mathcal B}}

\newcommand{\cE}{{\mathcal E}}

\newcommand{\cH}{{\mathcal H}}

\newcommand{\cK}{{\mathcal K}}
\newcommand{\cL}{{\mathcal L}}

\newcommand{\cP}{{\mathcal P}}

\newcommand{\id}{{\text{Id} }}
\newcommand{\Gr}{{\text{Gr} }}

\newcommand{\ch}{{\text{ch} }}

\newcommand{\spin}{{ \text{spin}}}
\newcommand{\spinc}{{ \text{Spin}^c }}

\newcommand{\ind}{{ \text{Ind} }}

\newcommand{\tr}{{ \text{Tr} }}

\newcommand{\ba}{\begin{eqnarray}}
\newcommand{\na}{\end{eqnarray}}
\newcommand{\ban}{\begin{eqnarray*}}
\newcommand{\nan}{\end{eqnarray*}}

\newcommand{\C}{{\mathbb C}}

\newcommand{\R}{{\mathbb R}}
\newcommand{\Z}{{\mathbb Z}}

\renewcommand{\thefootnote}{\fnsymbol{footnote}}

\makeatletter
\g@addto@macro{\endabstract}{\@setabstract}
\newcommand{\authorfootnotes}{\renewcommand\thefootnote{\@fnsymbol\c@footnote}}%
\makeatother

\makeatletter
\@namedef{subjclassname@2020}{%
  \textup{2020} Mathematics Subject Classification}
\makeatother

\title[]{Equivariant Toeplitz Index Theory \\ On Odd-Dimensional Manifolds With Boundary}



\begin{document}

%\title{Equivariant Dai-Zhang Toeplitz index theorem}

\begin{center}
%	\Large
%	\textbf{Equivariant Toeplitz Index Theory \\ On Odd-Dimensional Manifolds With Boundary} \\  \par \bigskip
    \vspace{-1cm}
	\maketitle
	
	\normalsize
	\authorfootnotes
	Johnny Lim\textsuperscript{1}\footnote[2]{Corresponding author.}, %\footnote{j.lim@usm.my}\textsuperscript{1},
	Hang Wang\textsuperscript{2} %\footnote{wanghang@math.ecnu.edu.cn}\textsuperscript{2}
	\par \bigskip
	
	\textsuperscript{1} \small{School of Mathematical Sciences, Universiti Sains Malaysia, Malaysia} \par
	\textsuperscript{2}School of Mathematical Sciences \& Shanghai  Key Laboratory of PMMP, \\ East China Normal University, Shanghai, China\par \bigskip
	
%	\today
\end{center}

\address{School of Mathematical Sciences, Universiti Sains Malaysia, Penang, Malaysia
}
\email{johnny.lim@usm.my}

\address{School of Mathematical Sciences \& Shanghai  Key Laboratory of PMMP, East China Normal University, Shanghai, China 200241
}
\email{wanghang@math.ecnu.edu.cn}

				%%%%%%%%%%%%%%%%%%   Abstract   %%%%%%%%%%%%%%%%%%%%
\vspace{-0.5cm}

\begin{abstract}
In this paper, we establish an equivariant version of Dai-Zhang's Toeplitz index theorem for compact odd-dimensional spin manifolds with even-dimensional boundary. 
\end{abstract}

\tableofcontents

				%%%%%%%%%%%%%%%%%%   Introduction   %%%%%%%%%%%%%%%%%%%%
\vspace{-1em}
\section{Introduction}

Let $X$ be a closed even dimensional spin manifold. There is a canonical Poincar\'e duality between $K$-theory and $K$-homology, where a complex vector bundle $E$ over $X$ representing a class in $K^0(X)$ is mapped to the twisted spin Dirac operator $D_E$ representing a class in $K_0(X)$, the dual of $K$-theory:
\[
K^0(X)\rightarrow K_0(X), \qquad [E]\mapsto [D_E].
\]
The twisted Dirac operator $D_E$ is a $\Z_2$-graded Fredholm operator $\begin{bmatrix}0 & D_E^-\\ D_E^+ & 0\end{bmatrix}$ with respect to $E\otimes S^+\oplus E\otimes S^-$, where $S=S^+\oplus S^-$ is the spinor bundle over $X$.
The fundamental class $[D]\in K_0(X)$ is represented as the dual of $K$-theory by taking the Fredholm index of the twisted Dirac operators:
\[
K^0(X)\rightarrow\Z, \qquad [E]\mapsto \ind(D_E).
\]
The Fredholm index $\ind(D_E^+)=\ker\dim(D^+_E)-\ker\dim(D^-_E)$ can be calculated by the Atiyah-Singer index formula \cite{atiyah1968index3} via geometric method
\[
\ind(D_E^+)=\int_X\hat{A}(X)\ch(E),
\]
where $\hat{A}(X)$ is the $\hat{A}$-class of $X$ and $\ch(E)$ is the even Chern character of $E.$ When $X$ is a closed odd dimensional spin manifold, the $K$-theoretic Poincar\'e duality
\[
K^1(X)\rightarrow K_1(X), \qquad [g]\mapsto [T_g],
\] 
is to associate each $K^1$-representative $g: X\rightarrow U(n)$ the Toeplitz operator $T_g=PgP,$ where $P$ is the projection given by $\frac{1+D|D|^{-1}}{2}$ that maps to the nonnegative part of the spectrum of $D$.
The fundamental class $[D]\in K_1(X)$ is represented by its dual 
\[
K^1(X)\rightarrow\Z, \qquad [g]\mapsto \ind(T_g).
\]
With the parametrix given by $T_{g^{-1}}=Pg^{-1}P,$ the operator $T_g$ is indeed Fredholm. The index of the Toeplitz operator $T_g$ is equal to the spectral flow from $D$ to $g^{-1}Dg$ in the sense of Atiyah-Patodi-Singer (APS) \cite{atiyah1976spectral3}, see also \cite{booss2012elliptic}. On the other hand, Baum-Douglas \cite{baum1982toeplitz} computes the index and obtain
\[
\ind(T_g)=\int_X\hat{A}(X)\ch(g),
\]
where $\ch(g)$ is the odd Chern character in the sense of Getzler \cite{getzler1993odd}.

From a purely differential geometric perspective, the formula for index pairing is motivated to be generalised to the case of manifolds with boundary.
Let $X$ be a compact spin even dimensional manifold with boundary $Y$, and $E$ a complex vector bundle over $X$. 
Assume that the metric near the boundary is of product type and $D_E^+$ takes the form $c(\frac{d}{du})(\frac{d}{du}+D_{Y})$  near the boundary, where $d/du$ is the inward vector, $c(\cdot)$ is the Clifford action and $D_Y$ is the corresponding twisted Dirac operator on $Y$.
By imposing the APS boundary condition on the domain of $D^+$, that is, for a section $s$ in the Sobolev space $H^1(X, E\otimes S^+)$ of sections with derivatives up to order 1 in $L^2,$ the projection of $s|_{Y}$ to the nonnegative part of the spectrum of $D_{Y}$ vanishes. Then, $D^+_E$ is a Fredholm operator whose index is calculated by the APS index theorem \cite{atiyah1975spectral1}
\[
\ind(D_E^+)=\int_X\hat{A}(X)\ch(E)-\frac{\eta(D_{Y})+\dim\ker D_{Y}}{2},
\]
where $\eta(D_{Y})$ is the eta invariant of the self-adjoint operator $D_{Y}$. 

When $X$ is a compact spin odd-dimensional manifold with boundary $Y=\partial X$, in analogy to the Baum-Douglas Toeplitz index, a self-adjoint operator is needed in the construction. Let $P_{>0}$ be the orthogonal projection from $L^2(S\otimes E|_{Y})$ to the positive spectrum of $D_Y$, together with the choice of a Lagrangian subspace $\cL$ of $\ker D_Y$. Then, the twisted Dirac operator $D_E,$ endowed with the \textit{modified} boundary condition $P(\cL):=P_{>0}+P_{\cL},$ becomes a self-adjoint operator. Assume $\cL$ is fixed and write $P^\partial =P(\cL).$
The Toeplitz index is calculated by Douglas-Wojciechowski \cite{douglas1991adiabatic} when $g|_{\partial X}$ is trivial and by Dai-Zhang \cite[Theorem 2.3]{dai2006index} in general:
\[
\ind(T_g^E)= \int_X\hat{A}(X)\ch(E)\ch(g)-\overline{\eta}_{DZ}(Y, g)+\tau_{\mu}(g P^\partial g^{-1}, P^\partial, \cP_{X^-}),
\]
where $X^-$ corresponds to the part of the manifold with deleted cylinder, $\cP_{X^-}$ is the corresponding Calder\'on projection, $\overline{\eta}_{DZ}$ is the reduced Dai-Zhang $\eta$-invariant defined in \cite[Definition 2.2]{dai2006index}, and $\tau_{\mu}$ is the Maslov triple index in the sense of Kirk-Lesch \cite{kirk2000eta}. 

In this paper, we establish the equivariant version of the Dai-Zhang Toeplitz index theorem for an odd-dimensional manifold with even-dimensional boundary, following closely the method developed in \cite{dai2006index}. Let $H$ be a compact group of isometries of $X$ preserving its spin structure, as well as that of the boundary $Y=\partial X$. For each $h\in H$, 
we obtain an equivariant index formula 
\[
\ind_h(T_g^E)= \int_{X^h}\frac{\hat{A}(X^h)\ch_h(E)}{\det^{\frac12}(1-he^{R^N})}\ch_h(g)-\overline{\eta}_h(Y, g)+\tau^h_{\mu}(gP^\partial g^{-1}, P^\partial, \mathcal{P}_{X^-}),
\]
where $X^h$ is the fixed point submanifold of $X;$ $\ind_h$ is the equivariant analytic index \eqref{eq:equivind};  $\overline{\eta}_h(Y, g)$ is the reduced equivariant Dai-Zhang $\eta$-invariant; and $\tau^h_{\mu}$ is the equivariant Maslov triple developed in \cite[Section 6]{limwang2}. This is the main result (Theorem \ref{mainthm}) of our paper, which is a generalisation of the equivariant Toeplitz index theorem from closed manifolds (cf. \cite{fang2005equivariant}) to manifolds with boundary. The equivariant Toeplitz index is in fact given by the equivariant spectral flow from $D^E_{P^{\partial}}$ to $g^{-1}D^E_{gP^{\partial}g^{-1}}g$. 
We adopt the approach developed by Dai-Zhang in \cite{dai2006index} by introducing an intermediate operator $D^{\psi, g}$ which agrees with $D^E_{P^{\partial}}$ on the boundary and with $g^{-1}D^E_{gP^{\partial}g^{-1}}g$ on the interior. 
\begin{figure}
  \centering
  \includegraphics[width=\textwidth]{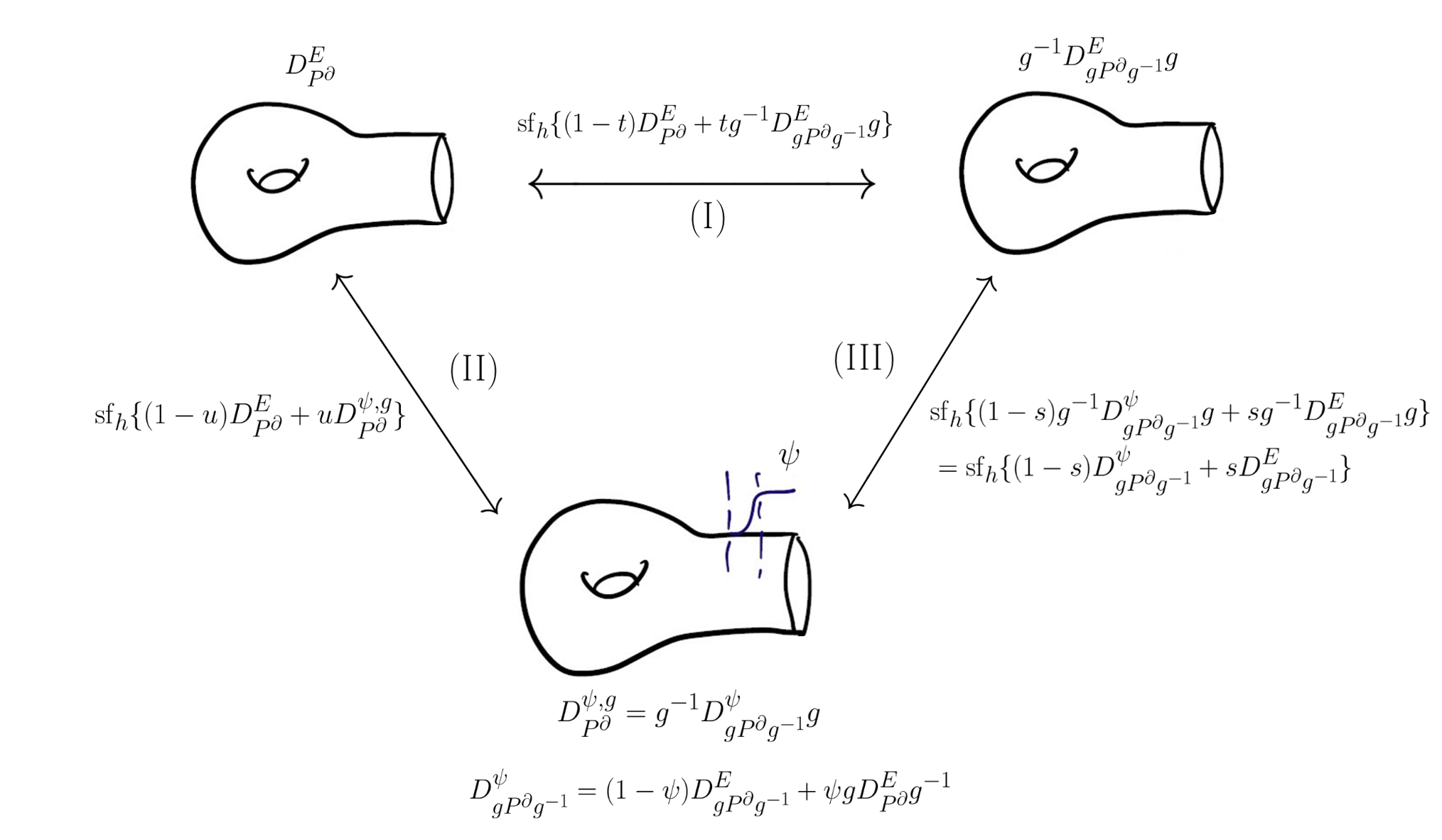}%\\
 \caption{Decomposition of spectral flows}\label{Fugure1}
\end{figure}
The spectral flow (I) is then decomposed into sum of spectral flows (II) and (III), cf. Figure~\ref{Fugure1}, where the local index calculation comes from (II), and (III) contributes to the boundary correction terms. 
The main result is established on the crucial equality in Lemma~\ref{lem:diff.ind}, which asserts that the difference between the perturbed and the unperturbed equivariant Toeplitz indices is some equivariant spectral flow. The main framework for the proof of the equality is equivariant \textit{higher} spectral flow, see Remark~\ref{hsfremark}.

The organisation of this paper is as follows. In Section \ref{sec2}, we start by discussing the equivariant higher spectral flow introduced in \cite{liu2021equivariant}, which is a generalisation of \cite{dai1998higher}. Then, we list the setup and define the required equivariant Toeplitz operator. In Section \ref{sec3.1}, we compute the equivariant Toeplitz index in terms of the equivariant spectral flow. Then, we introduce the equivariant even-type Dai-Zhang $\eta$-invariant in Section \ref{sec3.2} and establish relevant results in Section \ref{sec4.1}. Lastly, we present the full proof of our claim in Section \ref{sec4.2}.

\vspace{0.5cm}
\textbf{Acknowledgement.} The first author would like to express his utmost gratitude to Universiti Sains Malaysia for support. Part of the work was done while the first author visited the Research Center for Operator Algebras at East China Normal University back in October 2019. He is grateful to the center for financial support. The second author acknowledges the support from Science and Technology Commission  of Shanghai Municipality (STCSM), grant No.18dz2271000. We thank Professors Weiping Zhang and Bo Liu for their pre-review suggestions. Lastly, we also appreciate the referees very much for their helpful technical suggestions for improvement.\\

%%%%%%%%%%%%%%%%%%%%%%%%%%%%%%%%%%%%   

\section{Equivariant higher spectral flow} 
\label{sec2}

%%%%%%%%%%%%%%%%%%%%%%%%%%%%%%%%%%%%

Let $B$ be a compact manifold. Let $Z \to M \xrightarrow{\pi} B$ be a smooth fibration with the fiber $Z$ an odd-dimensional compact $\spin$ manifold (possibly with boundary). Let $H$ be a compact group of isometries of $B,$ where $h$ is an element of $H$ such that $H=\overline{\langle h \rangle}.$ For $t \in [0,1]$ and $b \in B,$ denote by $D_t=\{D_{b,t}\}$ a family of self-adjoint $H$-equivariant elliptic pseudodifferential operators on the fiber $Z$ and depends smoothly on $b \in B.$  In the following, we consider an $H$-equivariant analog of the classical spectral sections as defined in \cite[Definition 1]{melrose1997families,MR1484046} and \cite[Definition 1.2]{dai1998higher}. 

\begin{definition}\cite{liu2021equivariant}
Let $D=\{D_b\}_{b \in B}$ be a smooth family of first order elliptic pseudodifferential operators. An equivariant spectral section for $D$ is a continuous family of self-adjoint pseudodifferential projections $P=\{P_b\}_{b\in B}$ on the $L^2$-completion of the domain of $D$ for every $b\in B,$ which commutes with $h$ such that for some smooth functions $q:B \to \R$ and for every $b \in B$ we have
 \begin{equation}\label{eq:ss1}
 D_b u = \lambda u \implies \begin{cases}
 P_b u = u \quad \text{ if } \lambda > q(b) \\
 P_b u = 0  \quad \text{ if } \lambda < -q(b).
 \end{cases}
 \end{equation}
\end{definition}
This definition is a special case of \cite[Definition 2.2]{liu2021equivariant} for which an extra commutative condition is needed for even dimensional fibers. In \cite[Proposition 2.3(ii)]{liu2021equivariant}, Liu has shown that given two equivariant spectral sections $P$ and $Q,$ there exists an equivariant \textit{majorising} spectral section $R$ such that $PR=R$ and $QR=R.$  For every $b\in B$ and any such $R,$ $\mathrm{coker}\{P_bR_b: \im(R_b) \to \im(P_b)\}_{b \in B}$ forms an element $[R-P] \in K^0_H(B)$.  Moreover, the element $[P-Q]$ can be expressed as the virtual difference $[R-Q]- [R-P]$ which is independent of the choice $R.$ Then, given any equivariant spectral sections $P_1, P_2,$ and $P_3,$  we have the additive property 
\[
[P_3-P_1]=[P_3-P_2]+[P_2-P_1] \in K^0_H(B)
\] cf. \cite[\S 2.3]{liu2021equivariant}. In particular, \cite[Proposition 2.9(i)]{liu2021equivariant} asserts that every element of the equivariant $K$-theory can be written as the difference of classes comprising  pairs of equivariant spectral sections.

For the existence of equivariant spectral sections, we assume that the equivariant index bundle of $D_0$ vanishes, cf. \cite[Proposition 2.3(i)]{liu2021equivariant}. Then, by homotopy invariance, the rest of the equivariant index bundle for every $D_t$ also vanishes. Let $Q_0$ and $Q_1$ be equivariant spectral sections for $D_0$ and $D_1$ respectively. Let $\widetilde{D}=\{D_{b,t}\}_{b \in B, t \in I}$ be the equivariant total family parametrised by $B \times I.$ Then, there is a corresponding equivariant total spectral section $\widetilde{P}=\{P_{b,t}\}_{b \in B, t \in I}.$ Let 
\begin{equation} \label{eq:totalssres}
P_t=\widetilde{P}|_{B \times \{t\}}, \quad t \in I.
\end{equation}

\begin{definition} \label{higheqspdef1}
Let $P_t$ be as in \eqref{eq:totalssres}. The $H$-equivariant higher spectral flow of the pairs $(D_0,Q_0)$ and $(D_1,Q_1)$ is given by 
\begin{equation} \label{eq:higheqspdef1}
    \mathrm{sf}_H\{(D_0,Q_0),(D_1,Q_1)\} = [Q_1 -P_1] - [Q_0 -P_0] \in K^0_H(B).
\end{equation}
\end{definition}

Let $E \to M$ be an $H$-equivariant complex vector bundle. Let $\nabla^E$ be the $H$-invariant connection, i.e. the operator $\nabla^E$ from $\Gamma(M,E)$ to $\Omega^1(M,E)=\Gamma(M,\Lambda^1 T^*M \otimes E)$ commutes with the natural action of every element $h \in H$ on $\Omega^1(M,E).$ We also assume $E$ admits a Hermitian structure and $h$ preserves the Hermitian structure on $E.$ Assume $TZ\to M$ admits a spin structure and is fixed. Let $S(TZ)$ be the complex spinor bundle for $TZ$ and is assumed to commute with the $H$-actions. Let $\nabla^{S(TZ)}$ be the $H$-equivariant Hermitian connection on $S(TZ).$ Let $\nabla^{S(TZ)\otimes E}:= \nabla^{S(TZ)} \otimes 1 + 1 \otimes \nabla^E$ be the tensor product connection preserved by $H$-actions. For $b \in B,$ let $D^E_b: \Gamma((S(TZ) \otimes E)|_b) \to \Gamma((S(TZ) \otimes E)|_b)$ be the self-adjoint $H$-equivariant twisted Dirac operator defined by 
\begin{equation}\label{eq:twistedirac0}
D^E_b:=\sum_i c(e_i)\nabla^{S(TZ)\otimes E|_{b}}_{e_i}
\end{equation}
where $c$ is the Clifford action and $\{e_i\}$ is a local orthonormal frame of $TZ.$ 
Let $D^E=\{D^E_b\}_{b \in B}$ be the smooth family of the equivariant Dirac operators $D^E_b$ parametrised by $B.$ The following is a directly equivariant adaptation of \cite[Definition 1.6]{dai1998higher}. Compare with \cite[Definition 2.1]{liu2021equivariant}.

\begin{definition}
For $b \in B,$ let $D_b$ be an equivariant self-adjoint Dirac-type operators on $M.$ That is, for $b \in B,$ each $D_b$ is a first order self-adjoint $H$-equivariant differential operator having the same principal symbol as $D^E_b.$  The collection $D=\{D_b\}_{b \in B}$ is said to be a $B$-family of $H$-equivariant Dirac-type operators. 
\end{definition}

For our purposes, we shall equip the $B$-family $D$ of Dirac-type operators with some spectral sections. To ensure their existence, we make the following assumption, cf. \cite[Proposition 2.3]{liu2021equivariant}. See also \cite[Assumption 1.7]{dai1998higher} and \cite[Proposition 1]{melrose1997families}.

\begin{assumption} 
\label{assump1}
For the family $D^E=\{D^E_b\}_{b \in B},$ assume $\mathrm{Ind}_H(D^E)=0$ in $K^1_H(B).$ 
\end{assumption}

The following is the $H$-equivariant analog of \cite[Definition 1.8]{dai1998higher}.

\begin{definition}
Let $D$ be a $B$-family of $H$-equivariant Dirac-type operators. Define an $H$-equivariant \textit{generalised} spectral section $Q=\{Q_b\}_{b \in B},$ one for every $D_b,$ to be a continuous family of self-adjoint zeroth order pseudodifferential projections whose principal symbol coincides with that of an $H$-equivariant spectral section $P=\{P_b\}_{b \in B}$ for $D^E=\{D^E_b\}_{b \in B}.$
\end{definition}

Let $Q_1$ and $Q_2$ be two equivariant generalised spectral sections of $D.$ Consider the composition of spectral sections
\begin{equation}
    Q_{2,b}Q_{1,b} : \im (Q_{1,b}) \to \im(Q_{2,b}), \quad \forall\: b \in B. 
\end{equation} Then, $Q_2Q_1=\{Q_{2,b}Q_{1,b}\}_{b \in B}$ defines a continuous family of equivariant Fredholm operators over $B.$ By \cite{atiyah1971index4} and \cite{segal1968equivariant}, the analytic family Fredholm index of $Q_2Q_1$ coincides with the topological index of the equivariant vector bundle defined by the spectral sections $Q_1$ and $Q_2:$
\begin{equation} \label{eq:indbundle1}
    \ind(Q_2Q_1)=[Q_1-Q_2] \in K^0_{H}(B).
\end{equation}

Suppose $Q_1$ and $Q_2$ are homotopic, i.e. there is a continuous path of equivariant generalised spectral sections $\widetilde{Q}_s$ with $\widetilde{Q}_0=Q_1$ and $\widetilde{Q}_1=Q_2$ for $s \in [0,1].$ Let $Q_3$ be another equivariant generalised spectral section of $D,$ then by the homotopy invariance of the equivariant family index, we have $[Q_1-Q_3]=[Q_2-Q_3].$ On the other hand, the additivity of equivariant index bundles formed via equivariant generalised spectral sections still holds in $K^0_H(B)$: $[Q_1-Q_2]+[Q_2-Q_3]=[Q_1-Q_3].$ 
Let $\widetilde{R}_t$ and $\widetilde{R}'_t$ be two equivariant generalised spectral sections of the equivariant total family $\widetilde{D}$ parametrised by  $B \times I.$ Then, for $t \in I,$ we have a continuous path of equivariant Fredholm operators $R'_tR_t$ where $R_t=\tilde{R}|_{B \times \{t\}}$ and $R'_t=\tilde{R}'|_{B \times \{t\}}.$ By the homotopy invariance of the equivariant family index,
\begin{equation}\label{eq:homotopyeqsp1}
    [R_1-R'_1]=\ind(R'_1R_1)=\ind(R'_0R_0)=[R_0-R'_0],
\end{equation} which shows that the equivariant family index \eqref{eq:indbundle1} is independent of the choice of generalised spectral sections. In the following, we obtain a slight generalisation of \eqref{eq:higheqspdef1}.

\begin{proposition}
The $H$-equivariant higher spectral flow of the pairs $(D_0,Q_0)$ and $(D_1,Q_1)$ can also be computed via an equivariant generalised total spectral section $\tilde{R}$ of the equivariant total family $\widetilde{D}:$
\begin{equation}\label{eq:higheqspdef2}
    \mathrm{sf}_H\{(D_0,Q_0),(D_1,Q_1)\} = [Q_1 -R_1] - [Q_0 -R_0] \in K^0_H(B).
    \end{equation}
\end{proposition}

This proposition can be seen as the equivariant analog of \cite[Theorem 1.11]{dai1998higher}.

Let $g$ be a $K^1_H$-representative of $M,$ cf. \cite[Lemma 1.4]{liu2021equivariant}. Consider the conjugation
\[
D_{g,b}:=g_b D_b g^{-1}_b,\quad b \in B.
\] Here, $g_b$ is interpreted as an equivariant map $M 
\to U(N)$ with trivial $H$-actions on $U(N)$, which extends to an operator from $C^\infty(S(TZ_b) \otimes E_b \otimes \C^N|_{Z_b})$ to itself by acting on $C^\infty(S(TZ_b) \otimes E_b)$ as identity. Similarly, for $b \in B,$ the $H$-action on the bundle $S(TZ_b) \otimes E_b \otimes \C^N|_{Z_b}$ for all $b \in B$ defines a map on the space of sections. Let $D(t)=\{D_b(t)\}_{ t \in [0,1], b \in B}$ be a path of $B$-families of self-adjoint equivariant Dirac-type operators  where  $D_b(t)$ is given by
\[
D_b(t)=(1-t)D_b + tD_{g,b}, \quad t \in [0,1].
\] 
On the other hand, we have for every $b\in B$ an associated $H$-equivariant spectral section $P_b$ for $D_b$ and  $g_b P_b g^{-1}_b$ for $g_b D_b g^{-1}_b.$ Hence, we obtain a family of equivariant spectral sections $gPg^{-1}=\{g_b P_b g^{-1}_b\}_{b \in B}$ of $gDg^{-1}.$ Then, its associated equivariant higher spectral flow is given by
\begin{equation}\label{eq:higheqspdef4}
 \mathrm{sf}_H\{(D,P),(gDg^{-1},gPg^{-1})\} \in  K^0_H(B).   
\end{equation}

The equivariant higher spectral flow \eqref{eq:higheqspdef4} only depends on the symbol of $D.$ To see this, consider another $B$-family $D'$ of equivariant self-adjoint Dirac-type operators having the same symbol as the family $D.$ Let $P'$ be any spectral section of $D'.$ Choose the respective majorising spectral sections $R$ and $gRg^{-1}$ for both $P,P'$ and $gPg^{-1}, gP'g^{-1},$ then 
\begin{align*}
\mathrm{sf}_H \{(D,P)&,(gDg^{-1},gPg^{-1})\} -\mathrm{sf}_h \{(D',P'),(gD'g^{-1},gP'g^{-1})\} \\
&=([gPg^{-1}-gRg^{-1}]-[P-R]) - ([gP'g^{-1}-gRg^{-1}]-[P'-R]) \\
&=[gPg^{-1}-gP'g^{-1}] -[P-P']=0 \in K^0_H(B),
\end{align*} where the last equality follows from the computation \[
[gPg^{-1}-gP'g^{-1}]=g[P-P']g^{-1}=[P-P']
\] in $K^0_H(B)$ as $g$ is homotopic equivariantly to the identity $\id$ via unitary operators. In other words, \eqref{eq:higheqspdef4} is independent of the choice of the equivariant spectral section $P.$

\begin{remark}\label{hsfremark}
The notion and properties of equivariant higher spectral flow are recalled above as they serve as a general framework for proving Lemma~\ref{indspecf}-\ref{lem:diff.ind} in the special case $B=pt.$ These lemmata relate the equivariant spectral flow with equivariant Toeplitz index. In particular, Lemma~\ref{lem:diff.ind} is crucial in establishing the main result Theorem \ref{mainthm}. 
\end{remark}

%%%%%%%%%%%%%%%%%%%%%%%%%%%%%%%%%%%%%%%%%%%%%%%%%%%%%%

\section{Elements of equivariant Dai-Zhang Toeplitz index theorem}
\label{sec3}

%%%%%%%%%%%%%%%%%%%%%%%%%%%%%%%%%%%%%%%%%%%%%%%%%%%%%%

Let $M$ be an odd-dimensional compact oriented $\spin$ manifold with boundary $\partial M.$ Let $G$ be a compact group of isometries acting on $M.$ Fix an $h \in G$ and consider the closed subgroup $H$ of $G$ generated by $h$, i.e. $h \in H=\overline{\langle h \rangle}.$ In the following, we consider the case of $H$-actions on $M$ which also preserving $\partial M.$ Assume the $\spin$ structure on $M$ is $H$-equivariant and fixed. Then, the induced spin structure on the boundary $\partial M$ is also $H$-equivariant.
Let $E$ be an $H$-equivariant Hermitian vector bundle over $M,$ equipped with an $H$-invariant Hermitian connection $\nabla^E.$
Let $D^E$ be the $H$-equivariant twisted Dirac operator on $M$ defined similar to \eqref{eq:twistedirac0} (take $B=pt,$ and the $H$-equivariant spinor bundle is $S=S(TM)$).
Denote by $P_{>0}$ the orthogonal projection from $L^2(S\otimes E|_{\partial M})$ to the positive eigenspace of the Dirac operator $D^E_{\partial M}$ on $\partial M.$
Choose an $H$-equivariant Lagrangian subspace $\cL$ of $\mathrm{ker} (D^E_{\partial M}).$ (For its existence, see Appendix $A$.) Let $P_{\cL}$ be the corresponding orthogonal projection defined via kernel.
Then, we define the modified projection by 
\begin{equation}\label{eq:modproj1}
P^{\partial}(\cL):=P_{>0}+P_{\cL}.
\end{equation}
Henceforth, we assume $\cL$ is chosen and fixed, so we will denote this by $P^\partial := P^\partial(\cL)$ for simplicity. The association of $D^E$ with  $P^{\partial}$, in which we will denote by $D^E_{P^{\partial}}$ henceforth, is an $H$-equivariant self-adjoint first order elliptic differential operator having point spectrum with $\infty$ as a cluster point and with finite dimensional eigenspaces. The $H$-equivariance assumption on $D^E$ implies that $P^\partial$ is also $H$-equivariant.

Let $g: M\rightarrow U(n)$ be a continuous $H$-equivariant map for which $H$ acts trivially on $U(n)$. Let $M \times \C^n,$ or simply $\C^n,$ be the trivial $H$-equivariant complex vector bundle of rank $n$ over $M,$ endowed with the trivial Hermitian metric and Hermitian connection. Then, $g$ extends to act on $S \otimes E$ by identity and as a smooth automorphism on $\C^n.$ Moreover, $g$ is assumed to be of product structure over the cylinder $[0,1] \times \partial M \subset M,$ i.e. $g|_{[0,1] \times \partial M}:= \pi^*(g|_{\partial M})$ with the natural projection $\pi:[0,1] \times \partial M \to \partial M.$ Let $L^{2}(S\otimes E \otimes \C^n)$ be the space of $L^2$-sections of $S \otimes E \otimes \C^n$ and  $L^{2,+}_{P^{\partial}}(S\otimes E \otimes \C^n)$ be the direct sum of all non-negative eigenspaces of $D^E_{P^{\partial}}.$  Let 
\begin{equation}\label{eq:orgproj1}
P_{P^{\partial}} : L^{2}(S\otimes E \otimes \C^n) \to L^{2,+}_{P^{\partial}}(S\otimes E \otimes \C^n)
\end{equation}
be the orthogonal projection. Denote by $D^E_{gP^\partial g^{-1}}$ the twisted Dirac operator $D^E$ equipped with the conjugated boundary condition $gP^\partial g^{-1}.$ Let  $L^{2,+}_{gP^{\partial}g^{-1}}(S\otimes E)$ be the direct sum of all non-negative eigenspaces of $D^E_{gP^{\partial}g^{-1}}.$  Let 
\begin{equation}\label{eq:orgproj2}
P_{gP^{\partial}g^{-1}} : L^{2}(S\otimes E \otimes \C^n) \to L^{2,+}_{gP^{\partial}g^{-1}}(S\otimes E \otimes \C^n)
\end{equation}
be the orthogonal projection.

\begin{definition}
Define the Toeplitz operator by 
\begin{equation}\label{eq:toepop1}
T^E_g:=P_{gP^{\partial}g^{-1}}\circ g\circ P_{P^{\partial}}: L^{2,+}_{P^{\partial}}(S\otimes E\otimes\C^n)\rightarrow L^{2,+}_{gP^{\partial}g^{-1}}(S\otimes E\otimes\C^n).
\end{equation} It is an $H$-equivariant Fredholm operator, as the equivariant analog of \cite[Equation (2.7)]{dai2006index}.
\end{definition}

%%%%%%%%%%%%%%%%%%%%%%%%%%%%%%%%%%%%%%%%%%%%%%%%%%%%%%

\subsection{Equivariant Toeplitz index and equivariant spectral flow}
\label{sec3.1}

%%%%%%%%%%%%%%%%%%%%%%%%%%%%%%%%%%%%%%%%%%%%%%%%%%%%%%

The main aim of this paper is to find integral formula  of the equivariant index of the Toeplitz operator $T^E_g$ at $h\in H$ defined by
\begin{equation}\label{eq:equivind}
\mathrm{Ind}_h(T^E_g):=\mathrm{Tr}(h|_{\mathrm{ker}(T^E_g)})-\mathrm{Tr}(h|_{\mathrm{coker}(T^E_g)}) \: \in \C,
\end{equation}
serving as an odd analog of Donnelly's result in \cite{donnelly1978eta}.
This number is the character value at $h$ of the equivariant index
\[ 
\mathrm{Ind}_H(T_g^E):=[\mathrm{ker}(T^E_g)]-[\mathrm{coker}(T^E_g)]\in R(H). 
\]
In fact, given a finite dimensional representation $(\pi, V_{\pi})$ of $H$, let $\chi_h(x)=\mathrm{Tr}(h|_{V_{\pi}}).$ Then $\mathrm{Ind}_h(T^E_g)=\chi_h(\mathrm{Ind}_H(T^E_g)).$

Consider the path connecting the Dirac operators $(D^E,P^{\partial})$ and $(g^{-1} D^Eg,P^{\partial}),$ i.e. 
\begin{equation} \label{eq:path1}
(1-t)D^E_{P^{\partial}}+tg^{-1}D^E_{gP^\partial g^{-1}}g\qquad 0\le t\le 1.  
\end{equation}
Note that the second term follows from 
\begin{equation} \label{eq:bdryconjcond}
(g^{-1}D^E g, P^\partial) = g^{-1}(D^E, g P^\partial g^{-1}) g = g^{-1}D^E_{g P^\partial g^{-1}} g.
\end{equation} 
That is, the left hand side of \eqref{eq:bdryconjcond} denotes the conjugated operator $g^{-1}D^Eg$ equipped with the boundary projection $P^\partial,$ whilst the right side denotes the conjugation of the operator $D^E_{g P^\partial g^{-1}}$ (equipped with the boundary condition $g P^\partial g^{-1}$) by $g.$
For the equivariant spectral flow 
\[
\mathrm{sf}_H\{(1-t)D^E_{P^{\partial}}+tg^{-1}D^E_{gP^\partial g^{-1}}g\}\in R(H)
\]
apply the character value at $h$ as above, and we obtain 
\[
\mathrm{sf}_h\{(1-t)D^E_{P^{\partial}}+tg^{-1}D^E_{gP^\partial g^{-1}}g\}:=\chi_h[\mathrm{sf}_H\{(1-t)D^E_{P^{\partial}}+tg^{-1}D^E_{gP^\partial g^{-1}}g\}]\in \C.
\]

\begin{lemma} 
\label{indspecf}
For $h \in H,$ the following equality holds in $\C$:
\begin{equation} \label{eq:indspecf}
   \mathrm{Ind}_h(T^E_g)= -\mathrm{sf}_h \{(1-t)D^E_{P^{\partial}}+tg^{-1}D^E_{gP^\partial g^{-1}}g\}.
\end{equation}
\end{lemma}

\begin{proof}
%Let $h \in H.$ 
By taking $Q_1= P_{P^\partial}$ and $Q_2=g^{-1}P_{gP^\partial g^{-1}}g,$ 
and apply \eqref{eq:indbundle1} for $B=pt$ and using the fact that  $\ind_H(Q_2Q_1)=\ind_H(gQ_2Q_1),$ we obtain
\begin{equation}\label{eq:indspecf1}
\mathrm{Ind}_H(T^E_g)=[P_{P^{\partial}}-g^{-1}P_{gP^\partial g^{-1}}g] = -[g^{-1}P_{gP^\partial g^{-1}}g-P_{P^{\partial}}] \in R(H).
\end{equation} 
On the other hand, by considering the trivial $H$-equivariant paths $Q_t \equiv P_{P^{\partial}}$ and $g^{-1}Q_tg$ for all $t,$ we see that the right side of \eqref{eq:indspecf1} coincides with 
\begin{equation}\label{eq:indspecf2}
\mathrm{sf}_H\{D^E_{P^\partial}, g^{-1}D^E_{gP^\partial g^{-1}} g\}
= [g^{-1}P_{gP^\partial g^{-1}}g-P_{P^{\partial}}] - [P_{P^{\partial}}-P_{P^{\partial}}]
= [g^{-1}P_{gP^\partial g^{-1}}g-P_{P^{\partial}}]. 
\end{equation} 
Here, $\mathrm{sf}_H\{D^E_{P^\partial}, g^{-1}D^E_{gP^\partial g^{-1}} g\}=\mathrm{sf}_H \{(1-t)D^E_{P^{\partial}}+tg^{-1}D^E_{gP^\partial g^{-1}}g\}.$ Thus, combining \eqref{eq:indspecf1} and \eqref{eq:indspecf2}, together with taking $\chi_h$ on both of these equations, we obtain the desired result.
\end{proof}

Note that the two operators $g^{-1}D^E_{gP^{\partial}g^{-1}} g$ and $D^E_{P^\partial}$ are different on both the interior and the boundary of $M.$
The analysis becomes much more involved and complicated.
In \cite{dai2006index} Dai and Zhang avoided this problem by introducing an intermediate operator
\begin{equation}\label{eq:Dpsig1}
D^{\psi}_{gP^{\partial}g^{-1}}:=(1-\psi)D^E_{gP^{\partial}g^{-1}}+\psi gD^E_{P^{\partial}} g^{-1}
\end{equation}
where $\psi$ is a fixed smooth cutoff function on $M$ being $1$ in the $\epsilon$-neighbourhood of $\partial M$ and $0$ outside $2\epsilon$ neighbourhood of $\partial M.$ In our case, $D^{\psi}_{gP^{\partial}g^{-1}}$ is $H$-equivariant.
The spectral flow $-\mathrm{sf}_h\{(1-t)D^E_{P^{\partial}}+tg^{-1}D^E_{g P^\partial g^{-1}}g\}$ is then decomposed into the sum of two spectral flows:
\begin{equation}
\label{eq:sf1}
-\mathrm{sf}_h\{(1-u)D^E_{P^{\partial}}+u g^{-1}D^{\psi}_{gP^{\partial}g^{-1}}g\}, \quad 0\le u\le 1,
\end{equation}
and
\begin{equation}
\label{eq:sf2}
-\mathrm{sf}_h\{(1-s)g^{-1}D^{\psi}_{gP^{\partial}g^{-1}}g + s g^{-1}D^E_{g P^\partial g^{-1}} g\}, \quad 0\le s\le 1.
\end{equation}
Note that the operators at the ending points of the path in the spectral flow~(\ref{eq:sf1}) coincide on the boundary but different on the interior.

Next, we express the equivariant spectral flow (\ref{eq:sf1}) as the equivariant index of some perturbed $H$-equivariant Toeplitz operator. Following the notation as before, let $L^{2,+,\psi}_{gP^{\partial}g^{-1}}(S\otimes E \otimes \C^n)$ be the non-negative eigenspaces of the $H$-equivariant twisted Dirac operator $D^{\psi}_{gP^{\partial}g^{-1}}$ and denote by $P^{\psi}_{gP^{\partial}g^{-1}}$  the corresponding orthogonal projection. Let $P_{P^\partial}$ be as in \eqref{eq:orgproj1}.
Then, we define the $H$-equivariant perturbed Toeplitz operator by 
\begin{equation}\label{eq:toepop2}
T_{g, \psi}^E:=P^{\psi}_{gP^{\partial}g^{-1}}\circ g\circ P_{P^{\partial}}: L^{2, +}_{P^{\partial}}(S\otimes E\otimes\C^n)\rightarrow L^{2, +}_{gP^{\partial}g^{-1}}(S\otimes E\otimes\C^n).
\end{equation}
This is the equivariant analog of \cite[Equation (2.9)]{dai2006index}.

\begin{lemma}
\label{lem:Toepsisf}
For $h \in H,$ the following equality holds in $\C$:
\begin{equation}
\label{eq:Toepsisf}
\mathrm{Ind}_h(T^E_{g, \psi})=-\mathrm{sf}_h\{(1-u)D^E_{P^{\partial}}+u g^{-1}D^{\psi}_{gP^{\partial}g^{-1}}g\}.
\end{equation}
\end{lemma}

\begin{proof}

Consider $\mathrm{Ind}_H(T_{g, \psi}^E)=\mathrm{Ind}_H(g^{-1}T_{g, \psi}^E)$ together with \eqref{eq:indbundle1}, then by taking $Q_u:=(1-u)P_{P^{\partial}}+ug^{-1}P^{\psi}_{gP^{\partial}g^{-1}}g$ and $P_u \equiv P_{P^\partial},$  the desired equality \eqref{eq:Toepsisf} follows from the same approach in the proof of Lemma \ref{indspecf}. 
 
\end{proof}
%\vspace{-1em}
The second spectral flow (\ref{eq:sf2}) is stable under the conjugation of $g$. By reversing the direction of the path, (\ref{eq:sf2}) can be rewritten as 
\begin{equation}
-\mathrm{sf}_h\{(1-s)g^{-1}D^{\psi}_{gP^{\partial}g^{-1}}g + s g^{-1}D^E_{g P^\partial g^{-1}} g\} 
=\mathrm{sf}_h\{(1-s)D^E_{g P^\partial g^{-1}}  + s D^{\psi}_{gP^{\partial}g^{-1}} \}
\end{equation}
for $s \in [0,1]$ and $h \in H.$  Set
\begin{equation} \label{eq:Dpsig2}
D^\psi_{g P^\partial g^{-1}}(s):=(1-s) D^E_{g P^\partial g^{-1}}+sD^{\psi}_{gP^{\partial}g^{-1}}, \quad 0\le s\le 1.
\end{equation}
Then, we obtain the following lemma, which  essentially follows from the additivity of equivariant spectral flow, see \cite{limwang2}.

\begin{lemma} 
Let $T^E_{g}$ be the Toeplitz operator \eqref{eq:toepop1} and $T_{g, \psi}^E$ be the perturbed Toeplitz operator \eqref{eq:toepop2}. Then, for $h \in H,$
\label{lem:diff.ind}
\[
\mathrm{Ind}_h(T^E_g) -\mathrm{Ind}_h(T^E_{g,\psi}) =\mathrm{sf}_h\{ D^\psi_{g P^\partial g^{-1}}(s)\}_{s \in [0,1]}.
\]
\end{lemma}

\begin{proof} 
 By using the approach in the proof of Lemma \ref{indspecf}, we have
	\begin{equation}\label{eq:inddiff1}
	\ind_H(T^E_g)=\ind_H (P_{gP^\partial g^{-1}}gP_{P^\partial})=\ind_H (g^{-1}P_{gP^\partial g^{-1}}gP_{P^\partial})=\ind_H (P^g_{gP^\partial g^{-1}}P_{P^\partial})
	\end{equation} where $P^g_{gP^\partial g^{-1}}=g^{-1}P_{gP^\partial g^{-1}}g : L^2(S\otimes E \otimes \C^n) \to L^{2,+}_{gP^\partial g^{-1}}(S\otimes E \otimes C^n)$ is the equivariant orthogonal projection onto the space of the sum of eigenspaces with respect to the non-negative eigenvalues of the equivariant Dirac operator $D^E_{gP^\partial g^{-1}}.$ On the other hand, the equivariant index of the perturbed Toeplitz operator has a similar expression
	\begin{equation}\label{eq:pertinddiff1}
	\ind_H(T^E_{g,\psi})=\ind_H (P^{g,\psi}_{gP^\partial g^{-1}}P_{P^\partial})
	\end{equation} where $P^{g,\psi}_{gP^\partial g^{-1}}$ is an orthogonal projection from  $L^2(S\otimes E \otimes \C^n) $ to $L^{2,+,\psi}_{gP^\partial g^{-1}}(S\otimes E \otimes C^n).$ The difference \eqref{eq:inddiff1}-\eqref{eq:pertinddiff1} can be computed as follows.
	\begin{align*}
	\ind_H T^E_g- \ind_H T^E_{g,\psi}
%	&= \ind_h (P^g_{gP^\partial g^{-1}}P_{P^\partial}) - \ind_h (P^{g,\psi}_{gP^\partial g^{-1}}P_{P^\partial}) \\
	&= [P_{P^\partial}- P^g_{gP^\partial g^{-1}}] - [P_{P^\partial}- P^{g,\psi}_{gP^\partial g^{-1}}] 
	= [P^{g,\psi}_{gP^\partial g^{-1}}- P^g_{gP^\partial g^{-1}}] \\
	&= \mathrm{sf}_H\{g^{-1}D^E_{g P^\partial g^{-1}}g, g^{-1}D^{\psi}_{g P^\partial g^{-1}}g\} 
	= \mathrm{sf}_H\{ D^E_{g P^\partial g^{-1}}, D^{\psi}_{g P^\partial g^{-1}}\} \\
	&= \mathrm{sf}_H \{ D^\psi_{g P^\partial g^{-1}}(s)\}_{s \in [0,1]}.
	\end{align*} 
Then, the desired result follows by taking $\chi_h$ for all three terms $\ind_H T^E_g, \ind_H T^E_{g,\psi},$ and $\mathrm{sf}_H \{ D^\psi_{g P^\partial g^{-1}}(s)\}_{s \in [0,1]}$ respectively. 
\end{proof}

%%%%%%%%%%%%%%%%%%%%%%%%%%%%%%%%%%%%%%%%%%%%%%%%%%%%%%

\subsection{Equivariant Dai-Zhang $\eta$-invariant}
\label{sec3.2}

%%%%%%%%%%%%%%%%%%%%%%%%%%%%%%%%%%%%%%%%%%%%%%%%%%%%%%

The equivariant APS $\eta$-invariants are classically known and well-studied, see for example Donnelly's original definition \cite[pg 892]{donnelly1978eta}, Zhang \cite[Equation (0.1)]{zhang90equiveta},  Fang \cite[\S 3]{fang2005equivariant}, and Lim-Wang \cite{limwang2}.
For $h \in H$ and Re$(s) \gg 0,$ the $h$-equivariant $\eta$-function of a self-adjoint elliptic equivariant differential operator $D$ is defined by
\begin{equation} 
\label{eq:equiveta}
    \eta_h(D,s):=\sum_{\lambda \neq 0, \lambda \in \text{spec}(D)\backslash \{0\}} \tr (h^*_\lambda)\frac{\text{sgn}(\lambda)}{|\lambda|^s}
\end{equation} where $h^*_\lambda$ is the induced linear map on the $\lambda$-eigenspace.
Define $h$-equivariant $\eta$-invariant of $D$ by $\eta_h(D):= \eta_h(D,0),$ and its reduced version by 
\begin{equation} \label{eq:reducedeta1}
\overline{\eta}_h(D) =\frac{\eta_h(D)+ \tr(h|_{\ker(D)})}{2}.
\end{equation}
In the case of manifolds with  boundary, the regularity at $s=0$ of the equivariant $\eta$-function $\eta_h(D_{P},s)$ for $P$ in some appropriate projection space is non-trivial. A detailed discussion and proof is provided in Appendix B. 

We are now ready to define the equivariant version of the Dai-Zhang reduced $\eta$-invariant (cf. \cite[Definition 2.2]{dai2006index}) on  the even dimensional boundary $\partial M$ of $M.$
Let $g$ be a unitary representative in $K^1_H(\partial M)$. 
For $s\in [0,1],$ consider 
\begin{equation}\label{eq:Dpsigs1}
D^{\psi, g}_{[0,1]}(s):=sD^{\psi, g}_{[0,1]} +(1-s) g^{-1}D^Eg 
\end{equation}
which commutes with $h \in H$ for all $s$ and connects $g^{-1}D^Eg$ and $D^{\psi, g}_{[0,1]}:=g^{-1}D^{\psi}_{[0,1]}g$ over the cylinder $[0,1]\times \partial M.$ Such a cylinder is equipped with the equivariant boundary conditions $P^\partial$ at $\{0\} \times \partial M$ and $\id-g^{-1}P^\partial g$ at $\{1\} \times \partial M.$ Here, $D^{\psi}_{[0,1]}$ is the family $(1-\psi)D^E + \psi gD^Eg^{-1}$ over the cylinder $[0,1]\times \partial M$ with the same boundary condition. Equation \eqref{eq:Dpsigs1} can be simplified to 
\begin{equation}\label{eq:Dpsigs2}
   D^{\psi, g}_{[0,1]}(s) = D^E + (1-s\psi)g^{-1}[D^E,g], \quad s\in [0,1].
\end{equation} 

The equivariant $\eta$-function  and invariant associated to $D^{\psi, g}_{[0,1]}$ are respectively
\begin{equation}
\eta_h(D^{\psi, g}_{[0,1]}, s)=\sum_{\lambda\neq 0}\tr(h^*_\lambda)\frac{\mathrm{sgn}(\lambda)}{|\lambda|^s}, \quad \eta_h(D^{\psi, g}_{[0,1]}):= \eta_h(D^{\psi, g}_{[0,1]},0).
\end{equation} 
Similarly, from \eqref{eq:reducedeta1}, we define its associated reduced equivariant $\eta$-invariant by
\begin{equation}
\overline{\eta}_h(D^{\psi, g}_{[0,1]})=\frac{\eta_h(D^{\psi, g}_{[0,1]}) + \mathrm{Tr}( h|_{\ker( D^{\psi,g}_{[0,1]})})}{2}. 
\end{equation} 
%\vspace{-0.5cm}

\begin{definition} \label{equiDZeta1}
Let $h \in H.$ The equivariant Dai-Zhang $\eta$-invariant is defined by
\begin{equation} \label{eq:equiDZeta1}
\overline{\eta}_h(\partial M, g)=\overline{\eta}_h(D^{\psi, g}_{[0,1]})-\mathrm{sf}_h\{D^{\psi, g}_{[0,1]}(s)\}_{s\in[0,1]}
\end{equation}
where $D^{\psi, g}_{[0,1]}(s)$ is given by \eqref{eq:Dpsigs2}.
\end{definition}

%%%%%%%%%%%%%%%%%%%%%%%%%%%%%%%%%%%%%%%%%%%%%%%%%%%%%%%%%%%%%%%

\section{Proof of main theorem}
\label{sec4}
%%%%%%%%%%%%%%%%%%%%%%%%%%%%%%%%%%%%%%%%%%%%%%%%%%%%%%%%%%%%%%%

In this section, we present the full proof of the main theorem. First, in Section \ref{sec4.1} we establish the required equivariant boundary correction terms which involve equivariant Dai-Zhang $\eta$-invariant. Then, in Section \ref{sec4.2} we compute the corresponding local index formula using heat kernels. The main theorem is then established in Theorem \ref{mainthm} and \ref{mainthm2} for the spin and $\spinc$ cases respectively.

%%%%%%%%%%%%%%%%%%%%%%%%%%%%%%%%%%%%%%%%%%%%%%%%%%%%%%%%%%%%%%%

\subsection{Results concerning boundary correction terms}
\label{sec4.1}

%%%%%%%%%%%%%%%%%%%%%%%%%%%%%%%%%%%%%%%%%%%%%%%%%%%%%%%%%%%%%%%

Consider \eqref{eq:Dpsig1} with the variable $\psi$ and $1-\psi$ interchanged:  
\begin{equation}\label{eq:DpsigA}
D^{\psi}_{gP^{\partial}g^{-1}}=(1-\psi)D^E_{gP^{\partial}g^{-1}}+\psi (gD^Eg^{-1})_{gP^{\partial}g^{-1}}.
\end{equation}
Conjugate \eqref{eq:DpsigA} and set for $0\le u\le 1,$
\begin{equation}
\label{eq:Dpsigu2}
D^{\psi, g}_{P^{\partial}}(u)
=(1-u)D^E_{P^{\partial}}+ug^{-1}D^{\psi}_{gP^{\partial}g^{-1}}g 
=D^E_{P^{\partial}}+u(1-\psi)g^{-1}[D^E_{P^{\partial}}, g].
\end{equation}

Then, by the conjugation invariance of equivariant $\eta$-invariants, we obtain
\begin{equation} \label{eq:conjueta1}
    \overline{\eta}_h(D^{\psi, g}_{P^{\partial}}(1))
    =\overline{\eta}_h(D^{\psi, g}_{gP^{\partial}g^{-1}})
    =\overline{\eta}_h(D^{\psi}_{gP^{\partial}g^{-1}}).
\end{equation}

\begin{lemma} \label{lem:decomsf}
Let $h\in H.$ Let $D^{\psi,g}_{P^{\partial}}(s)$ be the path of perturbed operators as in \eqref{eq:Dpsigu2}. Then, as an operator with respect to $s,$ its equivariant spectral flow has an expression in terms of equivariant reduced $\eta$-invariants
\begin{equation}
\label{eq:decomsf}
\mathrm{sf}_h\{D^{\psi,g}_{P^{\partial}}(s)\}_{s\in[0,1]}
=\overline{\eta}_h(D^{\psi,g}_{P^{\partial}}(1))-\overline{\eta}_h(D^{E}_{P^{\partial}})-\int_0^{1}\frac{d}{ds}\overline{\eta}_h(D^{\psi,g}_{P^{\partial}}(s))ds.
\end{equation}
\end{lemma}

\begin{proof}
This follows from \cite[Theorem 9.5]{limwang2}.
\end{proof}

On the other hand, suppose $M$ is decomposed into $M^+ \cup_N M^-$ where 
\begin{equation}\label{eq:cutmfd1}
M^-:=M \backslash ([0,1] \times \partial M)
\end{equation} 
and $N=\{1\} \times \partial M$ is a closed hypersurface. We equip the boundary $\partial M^-=\{1\} \times \partial M$ with the boundary condition $P^{\partial M^-}.$ Then, we have 
\begin{equation} \label{eq:etasplit1}
    \overline{\eta}_h (D^E_{P^{\partial M}})= \overline{\eta}_h (D^E_{P^{\partial M^-}}).
\end{equation}
by a direct equivariant adaptation of \cite[Proposition 2.16]{muller1994eta}: 

\begin{proposition}
Let $M$ be a manifold with smooth boundary $Y$, $D$ a Dirac type operator on $M$. Let $a\ge 0$, $M_a=M\cup([0,a]\times Y)$ and $D(a)$ be the natural extension of $D$ on $M_a$. Let $D(a)_{\sigma}$ be the self-adjoint extension equipped M\"uller's boundary condition.
Then the equivariant eta invariant $\eta_h(0, D(a)_{\sigma})$ is independent of $a.$
\end{proposition}

\begin{proof}
The proof follows almost verbatim from the nonequivariant case except that one needs to show vanishing of $\eta_h(s, \hat D_a)$ for $\hat D_a=\gamma(\frac{\partial}{\partial u}+A)$ on $S_a\times Y$. 
Here $S_a$ stands for the circle of radius $2a$. In the nonequivariant case, this follows as a result of $\hat D_a$ having symmetric spectrum. 
In fact, let $\gamma\tau$ be the standard involution of the spinor bundle over $S_a\times Y$. Then $A\gamma\tau=\gamma\tau A$ (See~\cite{muller1994eta} equation (1.31)). 
An eigenfunction of $\hat D_a$ has the form 
$e_j\otimes\psi_i$ with eigenvalue $\lambda$ where $e_j$ and $\psi_i$ are eigenfunctions of $\gamma \frac{\partial}{\partial u}$ and $\gamma A$ respectively. 
Then it is straightforward to check that $\gamma\tau e_j\otimes \gamma\tau \psi_i$ is a corresponding eigenvector for eigenvalue $-\lambda$. 
Denote by the $\lambda$-eigenspace to be $E_{\lambda}$. Because $h$ acts by isometry so it preserves $\gamma$ and $\gamma\tau$, so 
\[
E_{\lambda}\rightarrow E_{-\lambda}, \qquad e_{j}\otimes \psi_i\mapsto \gamma\tau e_j\otimes \gamma\tau \psi_i
\]
is an equivariant map. 
Therefore $\mathrm{Tr} (h|_{E_{\lambda}})=\mathrm{Tr} (h|_{E_{-\lambda}})$. This implies that $\eta_h(s, \hat D_a)$ also vanishes identically.
\end{proof}

\begin{lemma} \label{lem:etasf}
Let $h\in H.$ Let $D^{\psi,g}_{P^{\partial}}(s)$ be the path of perturbed operators as in \eqref{eq:Dpsigu2}. Suppose $M=M^\pm$ as above. Let $\tau^h_{\mu}$ be the equivariant Maslov triple index of the triple of $h$-equivariant boundary projections $(\cP^{\psi}_{[0,1]}, P^{\partial}, \cP^E_{M^-})$ where $\cP^{\psi}_{[0,1]}$ is the Calder\'{o}n projection associated to $D^\psi$ on $[0,1]\times \partial M$ with $gP^\partial g^{-1}$ at $\{0\} \times \partial M;$ and  $\cP^E_{M^-}$ is the Calder\'{o}n projection on $M^-.$  Then
\begin{equation}
\label{eq:etasf}
\overline{\eta}_h(D^{\psi,g}_{P^{\partial}}(1))-\overline{\eta}_h(D^{E}_{P^{\partial}})
=\overline{\eta}_h(D^{\psi, g}_{[0,1]})-\tau^h_{\mu}(\cP^{\psi}_{[0,1]}, P^{\partial}, \cP^E_{M^-}).
\end{equation}
\end{lemma}

\begin{proof}
This follows from \eqref{eq:conjueta1} and \eqref{eq:etasplit1} that
\[
\overline{\eta}_h(D^{\psi,g}_{P^{\partial}}(1))-\overline{\eta}_h(D^{E}_{P^{\partial}}) 
=\overline{\eta}_h(D^{\psi}_{gP^{\partial}g^{-1}}) - \overline{\eta}_h (D^E_{P^{\partial M^-}}),
\] 
then, by \cite[Theorem 9.10]{limwang2} \footnote{Note that in \cite[Theorem 9.10]{limwang2} the ``cut'' manifold is obtained by cutting a closed manifold into two manifolds with boundary $M^\pm$ along $N$ (see also \cite[\S 5]{kirk2000eta}), whereas in our case \eqref{eq:cutmfd1} a cylinder is removed from a manifold with boundary. However, \cite[Theorem 9.10]{limwang2} still works here because when $\hat{\eta}_h$ splits, $\tau_h$ only depend on the canonical Calder\'on projectors $\cP^{\psi}_{[0,1]}$ on the cylinder $[0,1]\times \partial M$ and $\cP^E_{M^-}$ on the interior $M^-$ respectively, as well as $P^{\partial M^-}$ on $\partial M^-,$ for which $P^\partial$ plays the same role from the viewpoint of \eqref{eq:etasplit1}.}, we have
\[
\overline{\eta}_h(D^{\psi}_{gP^{\partial}g^{-1}}) - \overline{\eta}_h (D^E_{P^{\partial M^-}})
= \overline{\eta}_h(D^{\psi, g}_{[0,1]})-\tau^h_{\mu}(\cP^{\psi}_{[0,1]}, P^{\partial}, \cP^E_{M^-}).
\]
\end{proof}

\vspace{-0.5cm}

%%%%%%%%%%%%%%%%%%%%%%%%%%%%%%%%%%%%%%%%%%%%%%%%%%%%%%

\subsection{Computation of local index formula}
\label{sec4.2}

%%%%%%%%%%%%%%%%%%%%%%%%%%%%%%%%%%%%%%%%%%%%%%%%%%%%%%

This section is dedicated to the full proof of the main result of this paper: Theorem \ref{mainthm}. First, we compute the integral term in \eqref{eq:decomsf}
\begin{equation}
\label{eq:variofeta}
\int_0^{1}\frac{d}{du}\overline{\eta}_h(D^{\psi,g}_{P^{\partial}}(u))du
\end{equation}
by relating it to a heat trace followed by local index theorem.

For simplicity, denote by 
\[
B_u:=D^{\psi, g}_{P^{\partial}}(u)
\]
the family of self-adjoint elliptic operators of order one as in \eqref{eq:Dpsigu2}.

Denote by $M^h$ the fixed point submanifold of $M.$
Associated to each component, labelled by $i$, denote the dimension of the $i$-th component to be $m_i$ and $N_i$ the normal bundle associated to the component.
By adapting Proposition 2.6 in \cite{muller1994eta} to the equivariant setting, one has
\begin{equation}
\label{eq:etaheattr}
\frac{\partial}{\partial u}\eta_h(s, B_u)=\frac{s}{\Gamma(\frac{s+1}{2})}\int_0^{\infty}t^{\frac{s-1}{2}}\mathrm{Tr}_h\left[\big(\frac{d}{du}B_u\big) e^{-tB_u^2}\right]dt
\end{equation}
when  $\mathrm{Re}(s)>\mathrm{max}\{m_i-1\}$. The integral on the right hand side admits a meromorphic extension and at $s=0$ it has a simple pole whose residue can be calculated.
This suggests that the heat trace
\begin{equation}
\label{eq:extheattr}
\mathrm{Tr}_h\left(\frac{d}{du}(B_u) e^{-tB_u^2}\right)
\end{equation}
will be investigated. Here,  $\mathrm{Tr}_h(\cdots) := \mathrm{Tr}( h \cdots).$
Let us obtain explicit expressions of (\ref{eq:extheattr}).
On one hand, one has
\[
\frac{d}{du}D^{\psi, g}_{P^{\partial}}(u)=(1-\psi)g^{-1}[D^E, g]
\]
which is a pseudodifferential operator of order $0$. See (3.12) of \cite{dai2006index}.
Then, (\ref{eq:extheattr}) is
\[
\mathrm{Tr}_h\left(\frac{d}{du}(B_u) e^{-tB_u^2}\right)=\mathrm{Tr}\left[h(1-\psi)g^{-1}[D^E, g] e^{-tD^{\psi, g}_{P^{\partial}}(u)^2}\right].
\]
On the other hand, $\mathrm{Tr}_h\left(\frac{d}{du}(B_u) e^{-tB_u^2}\right)$ has an asymptotic expansion when $t$ is close to $0$.

Adapt the asymptotic expansion of~\cite{Grubb} to the equivariant case for $B_u=D^{\psi, g}_{P^{\partial}}(u)$ where $0\le u\le 1$, we have
\begin{multline}
\label{eq:asymex}
\mathrm{Tr}\left(h\frac{d}{du}(B_u) e^{-tB_u^2}\right)=\mathrm{Tr}\left(h(1-\psi)g^{-1}[D^E, g] e^{-tD^{\psi, g}_{P^{\partial}}(u)^2}\right) \\
\sim \sum_i\sum_{0\le k<m_i}c_{i,k-m_i}(u)t^{\frac{k-m_i}{2}}+\sum_{k\ge m_i}(c_{i,k-m_i}(u)+c_{i,k-m_i}'(u)\log t)t^{\frac{k-m_i}{2}}
\end{multline}
as $t\to 0.$
Proceed as in Corollary 2.9 of~\cite{muller1994eta} by computing the residue at $0$, 
we find that $\frac{\partial}{\partial v}\eta_h(s, B_v)$ is holomorphic at $s=0$ with
\begin{equation}
\label{eq:eta&c}
\frac{\partial}{\partial u}\eta_h(s, B_u)|_{s=0}=-\frac{2}{\sqrt{\pi}}\sum_i c_{m_i-1, i}(u).
\end{equation}
In fact, in (\ref{eq:etaheattr}) when $t>1$, the integral is a holomorphic function in $s$ and does not contribute to the residue, while for $0<t\le 1,$ the trace in the integrand can be replaced by the first few terms (when $k<m_i$) in the small time asymptotics~(\ref{eq:asymex}) because the remainder term is again holomorphic at $0$. 
Thus, $\frac{\partial}{\partial u}\eta_h(s, B_u)|_{s=0}$ is equal to the residue at $s=0$ of 
\[
\frac{s}{\Gamma(\frac{s+1}{2})}\int_0^1 t^{\frac{s-1}{2}}\sum_{i}\sum_{0\le k<m_i} c_{i, k-m_i}(u)t^{\frac{k-m_i}{2}}dt
\]
where all terms vanishes except when $n=m_i-1$ for the $i$-th component.
Thus, (\ref{eq:eta&c}) is obtained. A detailed argument of this sort can be found for example in Lemma 6.6 of~\cite{MR3500823}. 
With $B_u=D^{\psi, g}_{P^{\partial}}(u)$ for $0\le u\le 1$, we have
\[
\frac{d}{du}\overline{\eta}_h(D^{\psi, g}_{P^{\partial}}(u))=-\frac{1}{\sqrt{\pi}}\sum_i c_{m_i-1, i}(u).
\]

Hence, in order to find $\int_0^{1}\frac{d}{du}\overline{\eta}_h(D^{\psi,g}_{P^{\partial}}(u))du$, we only need to find $\sum_i\int_0^1 c_{m_i-1, i}(u)du.$
To calculate $\sum_ic_{m_i-1, i}(u)$, we multiply both sides of (\ref{eq:asymex}) by $t^{\frac12}$ and compare the constant term by letting $t\to 0_+$:
\begin{equation}
\label{eq:c_mi}
\sum_i c_{m_i-1, i}(u)=\lim_{t\to0+}\mathrm{Tr}\left(h\ t^{\frac12}(1-\psi)g^{-1}[D^E, g]\ e^{-tD^{\psi, g}_{P^{\partial}}(u)^2}\right).
\end{equation}
Then, we need only to calculate the right hand side of \eqref{eq:c_mi} by using the local index theory technique.
Denote by $\overline{D}$ the extension of the Dirac operator $D$ from $M$ to the double of $M$ 
and
$E_I(t)$ the associated heat operator $e^{-t\overline{D^{\psi,g}(u)}^2}$. 
Let $E_b(t)$ be the heat operator $e^{-t(D^{\psi, g}_{P^{\partial}})^2}$ on the half cylinder $\partial M\times[0,\infty)$.

Denote by $E_I(t,x,y)$ and $E_b(t,x,y)$ the respective Schwartz kernels of $E_I(t)$ and $E_b(t)$.
Recall that on the half cylinder $\partial M\times[-2\epsilon, \infty)$, $\psi$ is a smooth increasing function from $0$ to $1$ which is $0$ on $[-2\epsilon, -\epsilon]$ and $1$ when $x\ge0.$

\begin{lemma} 
For $h \in H,$
\begin{multline}
\label{eq:heattrint}
\lim_{t\to0}\mathrm{Tr}\left(h(1-\psi)g^{-1}[t^{\frac12}D^E, g]\ e^{-tD^{\psi, g}_{P^{\partial}}(u)^2}\right)\\
=\lim_{t\to0} \int_M\mathrm{tr}\left[(1-\psi(x))(hg(h^{-1}x))^{-1}(h[t^{\frac12}D^E, g](h^{-1}x))hE_I(t, h^{-1}x, x)\right]d\mathrm{vol}.
\end{multline}
Here, $\mathrm{tr}$ denotes the matrix trace.
\end{lemma}

\begin{proof}
Let $\psi_I$ be a smooth decreasing function from $1$ to $0$ where $\psi'$ is supported in $[\frac37\epsilon, \frac47\epsilon]$ and set $\psi_{b}=1-\psi_I.$
Let $\phi_I$ be a smooth decreasing function from $1$ to $0$ where $\phi'$ is supported in $[\frac17\epsilon, \frac27\epsilon]$ and $\phi_b$ be smooth increasing function from $0$ to $1$ where $\phi_b'$ is supported in $[\frac57\epsilon, \frac67\epsilon]$.
Note that they are invariant under $h$ as functions on $\partial M\times[0, \infty).$
Construct an approximate heat kernel
\[
E(t)=\phi_I E_I(t)\psi_I+\phi_bE_b(t)\psi_b.
\]
By the standard heat kernel estimate as well as the estimate on the half cylinder with APS type boundary condition, one has the estimates:
\begin{equation}
\label{eq:1}
|\partial_t^k\partial_x^lE_I(t,x,y)|\le C t^{-n-\frac12-k-\frac{l}{2}}e^{-\frac{d(x,y)^2}{4t}},\qquad |\partial_t^k\partial_x^lE_b(t,x,y)|\le Ct^{-n-\frac12-k-\frac{l}{2}}e^{-\frac{d(x,y)^2}{4t}}.
\end{equation}
Thus, from a similar estimate as in \cite[Lemma 22.11]{booss2012elliptic}, we have
\[
|(\partial_t+(D^{\psi,g})^2_x))E(t,x,y)|\le C e^{-Cd(x,y)^2/t}.
\]
In particular, this holds when $(x,y)$ is replaced by $(h^{-1}x,x)$, because $\phi_I,\phi_b$ satisfy $\phi_I(h^{-1}x)=\phi_I(x)$ and $\phi_b(h^{-1}x)=\phi_b(x).$
Also, by definition the support of $\phi_I'$ (resp. $\phi_b'$) and the support of $\psi_I$ (resp. $\psi_b$) are $\frac17\epsilon$-apart.
This shows the existence of $C>0$ such that
\begin{equation}
\label{eq:5}
|(\partial_t+(D^{\psi,g})^2_x))E(t,h^{-1}x,x)|\le C e^{-C/t}.
\end{equation}
By the standard procedure of Duhamel's principle, the relations between the approximate heat operator $E(t)$ and the true heat kernel of $e^{-t(D^{\psi, g})^2}$ can be established:
\begin{equation}
\label{eq:4}
e^{-t(D^{\psi,g})^2}=E(t)+\sum_{k=1}^{\infty}E(t)*C_k(t)
\end{equation}
where $C_1:=(\partial_t+(D^{\psi,g})^2)E(t)$ and $C_k=C_1*C_{k-1}$ is the convolution interation.
As $t$ approaches $0$, we have
\begin{equation}
\label{eq:3}
|e^{-t(D^{\psi,g})^2}(t,x,y)|<Ct^{-\frac{n}{2}}e^{-Cd(x,y)^2/t}
\end{equation}
and
\begin{equation}
\label{eq:2}
|e^{-t(D^{\psi,g})^2}(t,x,x)-E(t,x,x)|<Ce^{-C/t}.
\end{equation}
See \cite[Theorem 22.14]{booss2012elliptic}.
We will need to generalize (\ref{eq:2}) from $(x,x)$ to  $(h^{-1}x,x).$
First, when $h^{-1}x\neq x$, from (\ref{eq:1}) and (\ref{eq:3}) we have
\begin{equation}
\label{eq:6}
|e^{-t(D^{\psi,g})^2}(t,h^{-1}x,x)-E(t,h^{-1}x,x)|<Ct^{-\frac{n}{2}}e^{-Cd(h^{-1}x, x)^2/t}.
\end{equation}
From (\ref{eq:4}), every term in the series representing $e^{-t(D^{\psi,g})^2}(t,h^{-1}x,x)-E(t,h^{-1}x,x)$ contains $C_1$ and its iterations under the convolution with $E(t)$.
By applying (\ref{eq:5}) to each term of the form $X(t):=E(t)*C_k(t)$, we find that $(1-\phi(x))X(t, h^{-1}x, x)$ vanishes unless $d(h^{-1}x,x)\ge\frac17\epsilon$. Otherwise, it contradicts with $d(\mathrm{supp}(\phi_i'), \mathrm{supp}(\psi_i))\ge\frac{1}{7}\epsilon$.
In fact, let us look at $k=1$ for example.
There are four types of general terms:
\begin{align*}
&\int_M(1-\psi(x))\phi_b(x) E_b(t, h^{-1}x,y)\psi_b(y)\cdots\psi_I(x)dy\\
&\int_M(1-\psi(x))\phi_I(x)E(t,h^{-1}x,y)\psi_I(y) \phi_b'(y)\cdots \psi_b(x)dy\\
&\int_M\cdots\psi_I(y)\phi_I'(y)\cdots dy\\
&\int_M\cdots\psi_b(y)\phi_b'(y)\cdots dy.
\end{align*}
By the definition of the cutoff functions above, all terms vanish.
Therefore, we obtain
\[
\lim_{t\to0}\left(h(1-\psi)e^{-t(D^{\psi,g})^2}-h(1-\psi)E(t)\right)=0.
\]
Note that $[t^{-\frac12}D^E, g]\to dg$ as $t\to 0$, we have
\begin{align}
\label{eq:heattrint2}
\lim_{t\to0}&\mathrm{Tr}[hg^{-1}[t^{\frac12}D^E, g](1-\psi)e^{-t(D^{\psi,g})^2}] \nonumber \\
&=\lim_{t\to 0}[\mathrm{Tr}(h(1-\psi)g^{-1}[t^{\frac12}D^E, g]\phi_I E_{I}(t)\psi_I) 
+\mathrm{Tr}(h(1-\psi)g^{-1}[t^{\frac12}D^E, g]\phi_b E_{b}(t)\psi_b)].
\end{align}
The second term of \eqref{eq:heattrint2} vanishes because $(1-\psi)\phi_b=0$ by definition.
The lemma is then proved.
\end{proof}

Let
\[
A_t(u):=t^{\frac12}(1-u)\overline{D^E}+t^{\frac12}ug^{-1}\overline{D^{\psi}}g.
\]
Then $A_t(0)=t^{\frac12}D^E$, $A_t(1)=t^{\frac12}g^{-1}D^{\psi}g$ and from (\ref{eq:Dpsigu2}) we have
\begin{equation}
\label{eq:A_t}
A_t(u)=A_t(0)+u(1-\psi)g^{-1}[A_t(0), g].
\end{equation}
Note that the right hand side of (\ref{eq:heattrint}) is simply
\begin{equation*}
\mathrm{Tr}_h\left(\frac{d}{du}(A_t(u)) e^{-A_t(u)^2}\right).
\end{equation*}

\begin{lemma}
\label{prop:localind} For $h \in H,$ the equivariant heat trace \[
\sum_i c_{m_i-1, i}(u)=\lim_{t\to 0+}\mathrm{Tr}\left(h\frac{d}{du}(A_t(u)) e^{-A_t(u)^2}\right)
\]
is calculated by
\begin{equation*}
\sum_i\left( \frac{1}{2\pi \sqrt{-1}} \right)^{m_i+1} \int_{M^h_{(i)}}\frac{\hat{A}(M^h)}{\mathrm{det}^{\frac12}(1-he^{-R^{N_i}})}\ \mathrm{tr}\left[he^{-R^E}\right]
\ \mathrm{tr}\left[h\ g^{-1}dg\ e^{(1-u)u(g^{-1}dg)^2}\right].
\end{equation*}
\end{lemma}

\begin{proof}
From (\ref{eq:A_t}) one has
\begin{align*}
A_t(u)^2=&A_t(0)^2+[A_t(0), u(1-\psi)g^{-1}[A_t(0),g]]\\
&+u(1-\psi)g^{-1}[A_t(g), g]u(1-\psi)g^{-1}[A_t(0), g].
\end{align*}
As $t\to 0$, we have $[A_t(0), g]\to dg$ and $d(g^{-1})=-g^{-1}dg g^{-1}$, then
\begin{align*}
A_t(u)^2-A_t(0)^2\to& -u(d\psi)g^{-1}dg+u(1-\psi)d(g^{-1})dg+u(1-\psi)g^{-2}d^2g+u^2(1-\psi)^2(g^{-1}dg)^2\\
&=-[u^2(1-\psi)^2-u(1-\psi)](g^{-1}dg)^2-u(d\psi)g^{-1}dg.
\end{align*}
Similarly, as $t\to 0$,
\[
\frac{d}{du}(A_t(u))=(1-\psi)g^{-1}[A_t(0), g]\to (1-\psi)g^{-1}dg.
\]
Set the following matrix-valued function on $M$
\[
C_u:=\exp\left[(u^2(1-\psi)^2-u(1-\psi))(g^{-1}dg)^2+u\ d\psi\ g^{-1}dg\right].
\]
Then
\begin{align*}
&\lim_{t\to0+}\mathrm{Tr}\left(h\frac{d}{du}(A_t(u)) e^{-A_t(u)^2}\right)\\
=&\mathrm{Tr}\left[h(1-\psi)\ g^{-1}dg\ C_u\ (\lim_{t\to0}e^{-A_t(0)^2})\right]\\
=&\int_M\mathrm{tr}[(1-\psi(h^{-1}x))\ h g^{-1}dg(h^{-1}x)\ hC_u(h^{-1}x)\ \lim_{t\to 0}hE_I(t, h^{-1}x, x)]d\mathrm{vol}\\
=& \sum_i \left( \frac{1}{2\pi \sqrt{-1}} \right)^{m_i+1} \int_{M^h_{(i)}}\frac{\hat{A}(M^h)}{\mathrm{det}^{\frac12}(1-he^{-R^{N_i}})}\ \mathrm{tr}(he^{-R_E^2})\ \mathrm{tr}(h(1-\psi)\ g^{-1}dg\ C_u(x))d\mathrm{vol}.
\end{align*}
To show that
\[
\mathrm{tr}(h(1-\psi)\ g^{-1}dg\ C_u(x))=\mathrm{tr}\left[h\ g^{-1}dg\ e^{(1-u)u(g^{-1}dg)^2}\right]
\]
when $x\in M^h$, we proceed exactly as the argument in \cite[equation (3.34)]{dai2006index}.
In fact, by using the property $d\psi\wedge d\psi=0$ one has
\begin{align*}
C_u&=e^{(u(1-\psi)-u^2(1-\psi)^2)(g^{-1}dg)^2}\cdot e^{u\ d\psi\ g^{-1}dg}\\
&=e^{(u(1-\psi)-u^2(1-\psi)^2)(g^{-1}dg)^2}(1+u\ d\psi\ g^{-1}dg).
\end{align*}
Then $\mathrm{tr}(h(1-\psi)\ g^{-1}dg\ C_u(x))$ splits into two terms:
\begin{align} \label{eq:trsum1}
\mathrm{tr}\big[h (1-\psi)\ g^{-1}dg\ &e^{(u(1-\psi)-u^2(1-\psi)^2)\ (g^{-1}dg)^2}\big] \\
&+\mathrm{tr}\big[h (1-\psi)u\ d\psi\ (g^{-1}dg)^2\ e^{(u(1-\psi)-u^2(1-\psi)^2)(g^{-1}dg)^2} \big]. \nonumber 
\end{align}
The second term is a linear combination of terms of the form $c_{u,\psi}\mathrm{tr}[(h\ g^{-1}dg)^{2k}]$ for some $k\in\Z$ and $c_{u,\psi}\in\C.$
Since
\[
\mathrm{tr}[(h\ g^{-1}dg)^{n}]=\mathrm{tr}\left([(h\ g^{-1}dg)^{n-1}, h\ g^{-1}dg]\right)=0
\]
when $n$ is even, the second term of the sum \eqref{eq:trsum1} vanishes.
Regarding the first term when $(1-\psi)(x)\neq 1$, $x$ belongs to a piece of the cylinder, for which $g$ does not change in the normal direction. Hence, 
\[
h \ g^{-1}dg\ e^{(u(1-\psi)-u^2(1-\psi)^2)\ (g^{-1}dg)^2}
\]
does not have differentials in the normal direction. However, the cylindrical part of $M^h$ either looks like $(\partial M)^h\times(-\epsilon, 0)$ or the empty set. Thus, when $(1-\psi)(x)\neq 0$, it does not contribute to the integration, and so $\mathrm{tr}(h(1-\psi)\ g^{-1}dg\ C_u(x))$ reduces to \[
\mathrm{tr}\left[h \ g^{-1}dg\ e^{(u-u^2)\ (g^{-1}dg)^2}\right].
\]
The proof is now complete.
\end{proof}

For an equivariant unitary map $g : M \to U(N),$ the equivariant analog of the odd Chern character (by Zhang \cite[equation (1.50)]{zhang2001lectures}) is given by
\begin{equation} 
\label{eq:equivoddChernch}
\mathrm{ch}_h(g,d)= \left( \frac{1}{2\pi \sqrt{-1}} \right)^{n+1} \sum^\infty_{n=0}  \frac{n!}{(2n+1)!} \mathrm{tr}\left[h(g^{-1}dg)^{2n+1}\right].
\end{equation} See also \cite[Proposition 1.2]{getzler1993odd} for the original formulation by Getzler.
Then, one can calculate (\ref{eq:variofeta}) using Lemma~\ref{prop:localind} by integrating over $u.$ Together with \eqref{eq:equivoddChernch}, we obtain the following theorem.

\begin{proposition} For $h \in H,$
\[
 \int_0^{1}\frac{d}{du}\overline{\eta}_h(D^{\psi,g}_{P^{\partial}}(u))du
=\sum_i\left( \frac{1}{2\pi \sqrt{-1}} \right)^{m_i+1} \int_{M^h_{(i)}}\frac{\hat{A}(M^h)}{\mathrm{det}^{\frac12}\left(1-he^{-R^{N_i}}\right)}\ \mathrm{tr}\left[he^{-R^E}\right]
\ \mathrm{ch}_h(g,d).
\] 
\end{proposition}

Combining Lemma~\ref{lem:Toepsisf}, Lemma~\ref{lem:decomsf}, and Lemma~\ref{lem:etasf}, we obtain:

\begin{theorem} For $h \in H,$
\label{thm:ind.T.psi}
\begin{multline*}
\mathrm{Ind}_hT^{E}_{g, \psi}=-\sum_i\left( \frac{1}{2\pi \sqrt{-1}} \right)^{m_i+1}\int_{M^h_{(i)}}\frac{\hat{A}(M^h)}{\mathrm{det}^{\frac12}(1-h e^{-R^{N_i}})}\\
\times \mathrm{Tr}[h\exp(-R^E)]\mathrm{ch}_h(g, d)-\overline\eta_h(D^{\psi,g}_{[0,1]})+\tau^h_{\mu}(\cP^{\psi}_{[0,1]}, P^{\partial}, \cP^E_{M^-}).
\end{multline*}
\end{theorem}

Finally, we arrive at the main theorem of this paper.

\begin{theorem} 
\label{mainthm} 
For $h \in H,$
\begin{multline*}
\mathrm{Ind}_hT^{E}_{g}=-\sum_i\left( \frac{1}{2\pi \sqrt{-1}} \right)^{m_i+1} \int_{M^h_{(i)}}\frac{\hat{A}(M^h)}{\mathrm{det}^{\frac12}(1-h e^{-R^{N_i}})}\\
\times \mathrm{Tr}[h\exp(-R^E)]\mathrm{ch}_h(g, d)-\overline\eta_h(\partial M, g)+\tau^h_{\mu}(gP^{\partial}g^{-1}, P^{\partial}, \cP^E_{M^-}).
\end{multline*}
\end{theorem}

\begin{proof}
By using Lemma~\ref{lem:diff.ind} and then Theorem~\ref{thm:ind.T.psi}, we have 
\begin{align*}
\mathrm{Ind}_h(T^E_g)=&\mathrm{Ind}_h(T^E_{g,\psi})+\mathrm{sf}_h\{ D^\psi_{g P^\partial g^{-1}}(s)\}_{s \in [0,1]} \\
=&-\sum_i\left( \frac{1}{2\pi \sqrt{-1}} \right)^{m_i+1}\int_{M^h_{(i)}}\frac{\hat{A}(M^h)}{\mathrm{det}^{\frac12}(1-h e^{-R^{N_i}})} \mathrm{Tr}[h\exp(-R^E)]\mathrm{ch}_h(g, d)+ \\
&-\overline{\eta}_h(D^{\psi, g}_{[0,1]})+\tau^h_{\mu}(\cP^{\psi}_{[0,1]}, P^{\partial}, \cP^E_{M^-})+\mathrm{sf}_h\{ D^\psi_{g P^\partial g^{-1}}(s)\}_{s \in [0,1]}.
\end{align*}
By identifying equivariant spectral flow and equivariant Maslov (double) index (\cite[Theorem 9.7]{limwang2}), we have 
\begin{align*}
&\mathrm{sf}_h\{ D^\psi_{g P^\partial g^{-1}}(s)\}_{s \in [0,1]}=\mathrm{Mas}_h(\cP^{\psi}_{[0,1]}(s), \cP_{M^-}) \\
&\mathrm{sf}_h(D^{\psi, g}_{[0,1]}(s), 0\le s\le 1)=\mathrm{Mas}_h(\cP^{\psi}_{[0,1]}(s), P^{\partial}).
\end{align*}
On the other hand, we also obtain the equivalence between equivariant Maslov double and triple indices (\cite[equation (6.2)]{limwang2})
\begin{align*}
\tau^h_{\mu}(gP^{\partial}g^{-1}, &P^{\partial}, \cP^E_{M^-})-\tau^h_{\mu}(\cP^{\psi}_{[0,1]}, P^{\partial}, \cP^E_{M^-})\\
=&\mathrm{Mas}_h(\cP^{\psi}_{[0,1]}(s), \cP^E_{M^-})-\mathrm{Mas}_h(\cP^{\psi}_{[0,1]}(s), P^{\partial}).
\end{align*}
Thus, by \eqref{eq:equiDZeta1} and a simple algebraic calculation, we obtain 
\begin{align*}
\mathrm{Ind}_h(T^E_g)=&-\sum_i \left( \frac{1}{2\pi \sqrt{-1}} \right)^{m_i+1} \int_{M^h_{(i)}}\frac{\hat{A}(M^h)}{\mathrm{det}^{\frac12}(1-h e^{-R^{N_i}})} \mathrm{Tr}[h\exp(-R^E)]\mathrm{ch}_h(g, d)+ \\
&-\overline{\eta}_h(D^{\psi, g}_{[0,1]})+\tau^h_{\mu}(gP^{\partial}g^{-1}, P^{\partial}, \cP^E_{M^-})+\mathrm{sf}_h\{ D^{\psi,g}_{[0,1]}(s)\}.
\end{align*}
The proof is complete.
\end{proof}

\begin{remark}
Note that the equivariant Toeplitz index theorem for closed manifolds has been proved by Fang in \cite{fang2005equivariant}. Thus, Theorem \ref{mainthm} can be viewed as a generalisation to the case of manifolds with boundary. 
\end{remark}

Suppose $M$ is equivariant $\spinc$, so that the fixed point set is oriented. Denote by $L$ some line bundle associated to the spin$^c$ structure. Then, $L$ is also equivariant. The above settings and proofs extend and we obtain an equivariant Toeplitz index theorem on compact spin$^c$ odd-dimensional manifolds with boundary.
\begin{theorem} 
\label{mainthm2} 
For $h \in H,$
\begin{multline*}
\mathrm{Ind}_h(T^{E}_{g})=-\sum_i\left( \frac{1}{2\pi \sqrt{-1}} \right)^{m_i+1}\int_{M^h_{(i)}}\frac{\hat{A}(M^h)e^{\frac{c_{h,1}(L)}{2}}}{\mathrm{det}^{\frac12}(1-h e^{-R^{N_i}})}\\
\times \mathrm{Tr}[h\exp(-R^E)]\mathrm{ch}_h(g, d)-\overline\eta_h(\partial M, g)+\tau^h_{\mu}(gP^{\partial}g^{-1}, P^{\partial}, \cP^E_{M^-}).
\end{multline*}
where $c_{h,1}(L)/2=\mathrm{Tr}[h \exp(-R^L/2)]$ is half of the equivariant first Chern form.
\end{theorem}

\begin{remark}
Note that Theorem~\ref{mainthm} are independent of the choice of cutoff functions as $\overline{\eta}_h(\partial M,g)$ is. It follows almost verbatim of the proof of \cite[Proposition 5.1]{dai2006index}: consider two cutoff functions $\psi_1,\psi_2$ and define a smooth path $\psi_t=(2-t)\psi_1 + (t-1)\psi_2$ for $t \in [1,2].$ 
From \eqref{eq:Dpsigs2}, consider three paths:
\begin{align}
    D^{\psi_1,g}_{[0,1]}(s) &= D^E + (1-s\psi_1)g^{-1}[D^E,g], \label{eq:path1}\\
    D^{\psi_2,g}_{[0,1]}(s) &= D^E + (1-s\psi_2)g^{-1}[D^E,g], \label{eq:path2}\\
    D^{\psi_t,g}_{[0,1]}(1) &= D^E + (1-\psi_t)g^{-1}[D^E,g] = D^{\psi_t,g}_{[0,1]}  \label{eq:path3}.
\end{align}
Note that $D^{\psi_t,g}_{[0,1]}(0)$ is independent of $t.$ From the homotopy invariant and concatenation property of equivariant spectral flow, from \eqref{eq:path1}, \eqref{eq:path2}, and \eqref{eq:path3}, we have for $t\in [1,2],$
\begin{equation}\label{eq:rmk411a}
\mathrm{sf}_h\{D^{\psi_t,g}_{[0,1]}\} =  \mathrm{sf}_h\{D^{\psi_2,g}_{[0,1]}(s)\}  - \mathrm{sf}_h\{D^{\psi_1,g}_{[0,1]}(s)\}. 
\end{equation}
On the other hand, by the equivariant spectral flow formula \cite[Theorem 9.5]{limwang2}, we have 
\begin{equation}\label{eq:rmk411b}
  \bar{\eta}_h (D^{\psi_2,g}_{[0,1]}) - \bar{\eta}_h (D^{\psi_1,g}_{[0,1]}) 
  =\frac12 \int^2_{1} \frac{d}{dt} \bar{\eta}_h (D^{\psi_t,g}_{[0,1]}) dt + \mathrm{sf}_h\{D^{\psi_t,g}_{[0,1]}\}.
\end{equation}
Following the approach in Section \ref{sec4.2}, one computes that $\frac{d}{d t} D^{\psi_t,g}_{[0,1]} = (\psi_1-\psi_2)g^{-1}[D^E,g].$ The factor $(\psi_1-\psi_2)$ then carries over throughout the heat kernel computation and eventually causing the term in Lemma~\ref{eq:heattrint} to be zero because the first term of the limit \eqref{eq:heattrint2} reduces to the interior of manifold, and both $\psi_1$ and $\psi_2$ are by definition zero there. Thus, $\frac{d}{dt} \bar{\eta}_h (D^{\psi_t,g}_{[0,1]}) \equiv 0.$ Finally, by piecing \eqref{eq:rmk411a}, \eqref{eq:rmk411b}, and  
\eqref{eq:equiDZeta1} together, we obtain the desired conclusion.

\end{remark}

\begin{remark}

Application-wise, Theorem \eqref{mainthm} and Theorem \eqref{mainthm2} for the spin and $\spinc$ cases respectively, can be applied as the main framework for establishing the \textit{equivariant analytic Pontryagin duality}, which is a generalisation of \cite{lim2019analytic}. This is left as a future work. Such a duality, if established, is expected to describe the Aharonov-Bohm phenomenon in a special case under Type IIB String theory in physics.  

\end{remark}

%%%%%%%%%%%%%%%%%%%%%%%%%%%%%%%%%%%%%%%%%%%%%%%%%%%%%%%%%
%\vspace{0.5cm}

\appendix
\setcounter{secnumdepth}{0}
\section{Appendix A. Cobordism invariance of equivariant index}
\label{AppendixA}

%%%%%%%%%%%%%%%%%%%%%%%%%%%%%%%%%%%%%%%%%%%%%%%%%%%%%%%%%

In this appendix, the cobordism invariance of the $H$-equivariant index of Dirac operators is established, following the elegant arguments of Nicolaescu \cite{nicolaescu97cobordism} closely (almost \textit{verbatim}). This $H$-cobordism invariance ensures the existence of the $H$-equivariant Lagrangian subspace $\cL$ of $\ker(D^E_{\partial M})$ in \S 3. We shall spell out all settings below for completeness and to avoid possible notational confusion.  

Let $X$ be a compact oriented $(2n+1)$-dimensional Riemannian manifold $(X^{2n+1},g_X)$ with $g_X$ a Riemannian metric and with boundary $Y^{2n}=\partial X$ such that $g_X$ is a product metric near the boundary. Denote by $s$ the longitudinal coordinate on a collar neighborhood of $Y.$ Assume an $H$-action on $X$ where $H$ is the closed subgroup generated by an element $h$ of a compact group $G$ of isometries of $X,$ which also extends to the collar neighborhood by acting trivially in the $s$-direction. Let $\hat{c}: T^*X \to \mathrm{End}(\hat{\cE})$ be the Clifford multiplication where $\hat{\cE} \to X$ is a bundle of complex self-adjoint Clifford modules (cf. \cite{berline03heat}). Let $\cE=\hat{\cE}|_{Y}.$ Assume that $\mathrm{End}(\cE) \cong Cl(T^*Y) \otimes \C.$ Thus, $Y$ admits a $\spinc$ structure and $\cE$ is the associated bundle of complex spinors. Assume $H$ preserves the $\spinc$ structures on both $X$ and $Y$ respectively. 

Denote by $D: C^\infty(\hat{\cE}) \to C^\infty(\hat{\cE})$ a formally self-adjoint Dirac operator such that it takes the form $D=\hat{c}(ds) (\nabla_s+ A)$ over the collar neighborhood, where $A: C^\infty(\hat{\cE}|_{\{s\} \times Y}) \to C^\infty(\hat{\cE}|_{\{s\} \times Y})$ is the formally self-adjoint tangential operator independent of $s.$ By equipping $\hat{\cE}$ with an $H$-invariant connection, we may assume $D,$ and thus $A,$ are $H$-equivariant. (The twisted case by an $H$-equivariant complex vector bundle $E$ can be done accordingly.) Set $\gamma=\hat{c}(ds).$ Since both $D$ and $A$ are symmetric, we have the anticommutation property 
\begin{equation}\label{eq:anticomm1}
\{\gamma, A\}=\gamma A+ A\gamma=0.  \tag{A.1}
\end{equation}
It follows from \cite{nicolaescu97cobordism} that the anticommutation property implies 
\[
A=
\begin{pmatrix}
0 & A_- \\
A_+ & 0 
\end{pmatrix}
\]
where $A_\pm: C^\infty(\cE_{\pm}) \to C^\infty(\cE_{\mp})$ and $A_-=A_+^*.$ For every $h\in H,$ define the $h$-equivariant analytic index by $\ind_h(A)= \tr(h|_{\ker(A_+)}) - \tr(h|_{\ker(A_-)}),$ cf. \cite{berline03heat}. 

\begin{manualtheorem}{A}
\label{cobordism}
With the settings above, the cobordism invariance of the $H$-equivariant index asserts that $\mathrm{Ind}_h(A)=0$ for all $h\in H.$ 
\end{manualtheorem}

\begin{proof}
For $t \gg 0,$ denote by $X_t$ the manifold obtained by attaching the long cylinder $[0,t] \times Y$ on the `right' side. The $H$ action extends trivially to $X_t.$ Moreover, $\hat{\cE}$ and $D$ also extend naturally to $\hat{\cE}_t$ and $D_t$ respectively over $X_t.$ For every $r\geq 0,$ denote by $L^{r,2}$ the Sobolev space of distributions in $L^2$ with $L^2$-derivatives up to order $r.$ Set 
\[
\cK_t = \{u \in L^{\frac{1}{2},2}(\hat{\cE}_t)\:|\: D_t u=0\}.
\] 
Then, by \cite{booss2012elliptic} there exists a well-defined continuous restriction map $r_t: \cK_t \to L^2(\hat{\cE}_t|_{\partial  X_t}).$ The image $\Lambda_t:=r_t(\cK_t)$ is classically known as the Cauchy Data space. Since $h$ commutes with both $\hat{\cE}_t$ and $D$ for all $h\in H,$ $\Lambda_t$ is an $H$-equivariant closed subspace of $L^2(\cE).$ The $H$-equivariance of $A,$ i.e. the $H$-action preserves the property \eqref{eq:anticomm1}, the property $\Lambda^\perp_t = \gamma \Lambda_t$ established in \cite{booss2012elliptic} still holds.

Let $\cH_I$ be the closed subspace of $L^2(\cE)$ spanned by the eigenvectors of $A$ corresponding to eigenvalues in $I,$ for every interval $I \subset \R.$ It was shown by Nicolaescu \cite{nicolaescu95maslov} that there exist a real number $r \geq 0$ and an $A$-invariant subspace $L_\infty \subset \cH_{[-r,r]}$ such that 
\begin{equation}\label{eq:lagperp1}
L_\infty^{\perp}=\gamma L_\infty, \tag{A.2} 
\end{equation} 
and $\Lambda_t \to \Lambda_\infty=L_\infty \oplus \cH_{(-\infty,-r]}$ in the gap topology of Kato \cite{kato95pertubation}. Let $P_\infty$ be the orthogonal projection onto $\Lambda_\infty,$ which commutes with $H$-actions given that $\Lambda_\infty$ is $H$-equivariant. Let $R_\infty=2P_\infty -\id$ be the orthogonal reflection in $\Lambda_\infty.$ Since $P_\infty$ commutes with all $h\in H,$ so is $R_\infty.$ Then, \eqref{eq:lagperp1} is equivalent to the anticommutation property 
\begin{equation}\label{eq:anticomm2}
R_\infty\gamma = - \gamma R_\infty.   \tag{A.3}
\end{equation} 
Since $\cH_{[-r,r]}$ is $\gamma$-invariant by \eqref{eq:anticomm1}, $\cH_{[-r,r]}=\cH^+_r \oplus \cH^-_r$ where $\cH_r^{\pm}$ is the $\pm 1$-eigenspace of $\sqrt{-1}\gamma$ on $\cH^{\pm}_{r}.$ Moreover, $\cH_{[-r,r]}$ is $H$-equivariant since $A$ is assumed to be. Then, \eqref{eq:anticomm2} implies that $R_\infty(\cH^{\pm}_r)=\cH_r^{\mp},$ which is an isomorphism. The $H$-equivariance of $\cH^{\pm}_r$ is preserved under this reflection. Since $\ker(A_\pm) \subset \cH^{\pm}_r$ and since $L_\infty$ is $H$-equivariant and $A$-invariant, the isometry between $\ker(A_+)$ and $\ker(A_-)$ is implemented by $R_\infty.$ This means that $\tr(h|_{\ker(A_+)})$ and $\tr(h|_{\ker(A_-)})$ coincide and thus $\ind_h(A)=0.$    
\end{proof}

%%%%%%%%%%%%%%%%%%%%%%%%%%%%%%%%%%%%%%%%%%%%%%%%%%%%%%%%%
%\vspace{0.5cm}

\appendix
\setcounter{secnumdepth}{0}
\section{Appendix B. Regularity of equivariant $\eta$-functions at $s=0.$}

%%%%%%%%%%%%%%%%%%%%%%%%%%%%%%%%%%%%%%%%%%%%%%%%%%%%%%%%%
In this appendix, we show the regularity of equivariant $\eta$-functions at $s=0$ by adopting an elegant method by Wojciechowski in \cite{wojciechowski1999zeta}.

Let $M$ be an odd-dimensional compact spin manifold with boundary $\partial M.$  Assume the collar $N =[0,1]\times \partial M$ admits a metric of product type. Then, the operator $D$ has the form $\gamma(\partial_t + A)$ over $N.$ Let $\cH$ be a complex separable Hilbert space with a fixed basis. For a closed subgroup $H$ of a compact group $G$ of isometries on $M$ (as described in Section 3), let $\cH'=L^2(H,\cH)$ be the space of square integrable functions on $H$ with values in $\cH.$ Let
\begin{align*}
\Gr^*_h(\cH')&=\{P \in \cB(\cH') \mid P^2=P=P^*, \gamma P\gamma^*=\id-P, P-P_{\geq} \text{ operator of order -1}\}, \\
\Gr^*_{h,\infty}(\cH')&=\{P \in \Gr^*_h(\cH') \mid P-P_\geq \text{ has smooth kernel}\}.
\end{align*}
In particular, the modified boundary projection $P^\partial$ defined in \ref{eq:modproj1} belongs to  $\Gr^*_{h,\infty}(\cH').$
Recall from \cite[Lemma 9.2]{limwang2} that the equivariant $\eta$-function $\eta_h(D,s)$ is the Mellin transform of $\tr(hDe^{-tD^2})$ for $h \in H,$ i.e. 
\begin{equation}
    \eta_h(D,s) = \frac{1}{\Gamma \left(\frac{s+1}{2} \right)} \int_0^\infty t^{\frac{s-1}{2}} \tr(hDe^{-tD^2}) dt.      \tag{B.1}
\end{equation}
This is the heat kernel representation of the equivariant $\eta$-function of $D$. This allows us to write the equivariant $\eta$-invariant of $D$ as 
\begin{equation}
   \eta_h(D)= \eta_h(D,0) = \frac{1}{\sqrt{\pi}} \int_0^\infty \frac{1}{\sqrt{t}} \tr(hDe^{-tD^2}) dt.         \tag{B.2}
\end{equation}

Denote by $D_P$ the operator $D$ equipped with the boundary projection $P \in \Gr^*_{h,\infty}(\cH').$ We aim to show the following claim.

\begin{manualtheorem}{B.1} \label{thmB.1}
For any $P \in \Gr^*_{h,\infty}(\cH'),$ there exists a positive constant $c>0$ such that for any $t \in [0,1],$ we have the estimate 
\begin{equation}
    |\tr(hD_P e^{-tD^2_P})| < c.    \tag{B.3}
\end{equation}
\end{manualtheorem} 

Note that the integral in (B.3) converges over $[1,\infty)$ because $D_P$ is self-adjoint and has discrete spectrum. Hence, 
Theorem \ref{thmB.1} regarding the small time estimate, once established, will be sufficient to imply that the function $\eta_h(D_P,s)$ is a holomorphic function of $s$ in the half-plane $\mathrm{Re}(s)>-1$ for any $P \in \Gr^*_{h,\infty}(D).$ Thus, the equivariant $\eta$-function $\eta_h(D_P,s)$ is regular at $s=0$ for any $P \in \Gr^*_{h,\infty}(D).$ In particular, as our special interest in $\S 3$, this regularity holds when $P=P^\partial.$  

The work of Wojciechowski in \cite{wojciechowski1999zeta} is extensive and highly technical. Thus, we shall only specify necessary modifications in equivariant settings and refer readers to his paper for full detail and application.   

Consider the path $\{g_t\}_{t\in [0,1]}$ formulated in Lemma~1.2 of \cite{wojciechowski1999zeta}, where $g_t$ is constant on $[0,1/4]$ and on $ [3/4,1].$   Define a unitary operator $U$ on $L^2(M,S)$ that acts as constant $U=g_t$ on the collar $N$ and as identity $U=\id$ on $M\backslash N.$ Since $h \in H$ commutes with $D$ and $U,$ one can show that $D_P$ is unitarily equivalent to  $D_{U,\partial}=(U^{-1}DU)_{P^\partial}$, cf. \cite[Lemma 1.3]{wojciechowski1999zeta}. In particular, over the collar $[0,1/4] \times \partial M,$ one obtains
\begin{equation}\label{eq:unitB1}
U^{-1}DU=D+K_1, \quad U^{-1}D^2U = D^2 + K_2,  \tag{B.4}
\end{equation} 
where $K_1=\gamma U^{-1}[A,U]$ and $K_2=U^{-1}[A^2,U].$ One verifies that $K_1$ anticommutes and $K_2$ commutes with $\gamma,$ and furthermore $hK_i=K_ih$ for $i=1,2.$ By unitary equivalence, the operator $D_{U,\partial}$ can be studied in place of $D_P.$ The representation \eqref{eq:unitB1} can then be used to construct the parametrix of the kernel of the operator $D_{U,\partial}e^{-tD^2_{U,\partial}}.$ The parametrix is built by gluing the heat kernel on the double manifold $\widetilde{M}$ and the heat kernel on the cylinder. The objects $S, D, U,$ and $h$ on $M$ can be extended to $\widetilde{S}, \widetilde{D},\widetilde{U},$ and $\widetilde{h}$ on $\widetilde{M}$ respectively. Consider the conjugation $\widetilde{U}^{-1}\widetilde{D}\widetilde{U} : C^\infty(\widetilde{M},\widetilde{S}) \to C^\infty(\widetilde{M},\widetilde{S}).$ Then, $\widetilde{U}^{-1}\widetilde{D}\widetilde{U}$ is unitarily equivalent to $\widetilde{D}.$ Moreover, $\widetilde{U}^{-1}\widetilde{D}\widetilde{U}e^{-t(\widetilde{U}^{-1}\widetilde{D}\widetilde{U})^2}$ $=\widetilde{U}^{-1}\widetilde{D}e^{-t\widetilde{D}^2}\widetilde{U}.$ 
By Duhamel's Principle, on $M\backslash N$ up to exponentially small error in $t,$ the kernel $\widetilde{E}_U(t;x,y)$ of $\widetilde{U}^{-1}\widetilde{D}\widetilde{U}e^{-t(\widetilde{U}^{-1}\widetilde{D}\widetilde{U})^2}$  is equal to the kernel $E_{U,\partial}(t,x,y)$ of $D_{U,\partial} e^{-t D^2_{U,\partial}}$ for $t \in (0,1).$ Hence we have:

\begin{manuallemma}{B.2} \label{lemmaB.2}
Let $M_{1/8}= M \backslash ([0,1/8] \times \partial M).$ For any $\epsilon>0,$ there exists a positive constant $C$ such that %for any $x \in M_{1/8}$ and 
for any $t \in (0,1),$ 
the estimate
\begin{equation}
    \int_{M_{1/8}}\| E_{U,\partial} (t,x,hx) - \widetilde{E}_U(t,x,hx)\|dx \leq C(e^{-\frac{C}{t}}+\epsilon)
    \tag{B.5}
\end{equation}
holds. 
\end{manuallemma}

\begin{proof} 
Following the lines from \cite[Theorem 1.2, 2.4, 4.1]{douglas1991adiabatic}, one has positive constants $c_1, c_2, c_3$ independent of $h, x$ such that  
\[
\|E_{U,\partial}(t,x,hx)\|<c_1e^{c_2t}t^{-\frac{n+1}{2}}e^{-c_3\frac{d(x,hx)^2}{t}}
\]
and 
\[
\|\widetilde E_{U}(t,x, hx)\|<c_1e^{c_2t}t^{-\frac{n+1}{2}}e^{-c_3\frac{d(x,hx)^2}{t}}.
\]
Then, when $x\neq hx$,
\begin{equation}
\label{xneqhx}
\|E_{U,\partial}(t,x, hx)-\widetilde E_{U}(t,x, hx)\|\le 2c_1e^{c_2t}t^{-\frac{n+1}{2}}e^{-c_3\frac{d(x,hx)^2}{t}}. \tag{B.6}
\end{equation}
When $x=hx,$ from Lemma 1.4 of \cite{wojciechowski1999zeta}, there exists $c_1',c_2'>0$ independent of $x$ such that 
\[
\|E_{U,\partial}(t,x, hx)-\widetilde E_{U}(t,x, hx)\|\le c_1'e^{-\frac{c_2'}{t}}.
\]
For every $\epsilon>0,$ by the continuity of the $E_{U,\partial}, \widetilde E_U$ in $x$, there exists $\delta>0$ such that for $y$ satisfying $d(y, M^h)<\delta$ so that 
\[
\|E_{U,\partial}(t,y, hy)-\widetilde E_{U}(t,y, hy)\|\le c_1'e^{-\frac{c_2'}{t}}+\epsilon.
\]
Also by the continuity of the distance function, the set $\{d(y,hy)^2\in M \mid d(y, M^h)\ge \delta\}\subset(0,\infty)$ is compact and hence have a positive lower bound $\delta'.$
Together with (\ref{xneqhx}), we see that there exists $d_1,d_2>0$, for all  $d(y,M^h)>\delta$, and $t\in(0,1)$, we have 
\[
\|E_{U,\partial}(t,x, hx)-\widetilde E_{U}(t,x, hx)\|\le d_1e^{-\frac{d_2\delta'}{t}}.
\]
Use the estimates above to integrate $\|E_{U,\partial}(t,x, hx)-\widetilde E_{U}(t,x, hx)\|$ in and outside the $\delta$-neighborhood of $M^h$, the result can be obtained. 
\end{proof}

In the rest of the appendix, we shall obtain an estimate for the kernel of $D_{U,\partial} e^{-tD^2_{U,\partial}}$ restricted to the collar $N$.
Because Duhamel's Principle can be applied again to replace the kernel $E_{U,\partial}(t,x,y)$ of $D_{U,\partial} e^{-tD^2_{U,\partial}}$ on $N$ by the corresponding kernel for $\gamma(\partial_u+A+K_1)e^{-t(-\partial_u^2+A^2+K_2)_{P^{\partial}}}$ on the infinite cylinder $Z:=[0,\infty)\times \partial M$, we only need to focus on this operator on $Z$ below. 
 
 Consider the operator $D_\partial=\gamma(\partial_u+A)_{P^\partial}$ on $Z.$ Then, its square can be computed to be $(-\partial^2_u+A^2)_{P^\partial},$ which is an unbounded self-adjoint operator in $L^2(Z,S).$ By adding $K_2$ from \eqref{eq:unitB1} to form the operator $(-\partial^2_u+A^2+K_2)_{P^\partial}$, the semigroup 
\begin{equation}\label{eq:et1}
e(t):= e^{-t(-\partial^2_u+A^2+K_2)_{P^\partial}}  \tag{B.7}
\end{equation} 
is well-defined, cf. \cite{wojciechowski1999zeta} and references therein. 
The heat operator $e(t)$ is closely related to $e_1(t):=e^{-t(-\partial^2_u + A^2)_{P^\partial}}$ by \cite[\S 22C]{booss2012elliptic}:
\begin{equation}\label{eq:et2}
    e(t)= e_1(t) + \sum^n_{n=1} \{e_1 \ast_n K_2e_1\}(t), \tag{B.8}
\end{equation}
where $\ast_n$ denotes the $n$-times convolution, i.e., $\ast_{n} K_2e_1= K_2e_1 \ast \cdots \ast K_2e_1.$

From \cite[Appendix A]{wojciechowski1999zeta}, by studying the classical heat operator $e^{-t(-\partial^2_u+A^2)_{P^\partial}}$ restricted to the eigenspace associated to each eigenvalue $\mu_n$ of $A$,  the eigenfunction in the direction normal to the boundary has the form 
\begin{equation}\label{eq:gn1}
g_n(t,u,v) = \frac{e^{-\mu^2_nt}}{2\sqrt{\pi}t}\left\{ e^{-\frac{(u-v)^2}{4t}} - e^{-\frac{(u+v)^2}{4t}} \right\}, \quad \text{ for } n>0,  \tag{B.9}
\end{equation}
and 
\begin{align*}
g_n(t,u,v) = &\frac{e^{-(-\mu)^2_nt}}{2\sqrt{\pi}t}\left\{ e^{-\frac{(u-v)^2}{4t}} - e^{-\frac{(u+v)^2}{4t}} \right\}, \label{eq:gn2} \tag{B.10} \\
 &+ (-\mu_n) e^{-(-\mu_n)(u+v)} \mathrm{erfc}\left( \frac{u+v}{2\sqrt{t}} - (-\mu_n)\sqrt{t} \right)\quad \text{ for } n<0,  \nonumber
\end{align*}
where $\mathrm{erfc}(x) =\displaystyle \frac{2}{\sqrt{\pi}} \int^\infty_x e^{-r^2}dr.$
Then the kernel $E_1(t,(u,y),(v,z))$ of $e_1(t)$ can be explicitly represented by
\begin{equation}\label{eq:kerE1}
E_1(t,(u,y),(v,hy)) = \sum_{n\in \Z \backslash \{0\}} g_n(t,u,v) \varphi_n(y) \otimes h^*\varphi_n^*(hy), \quad h \in H. \tag{B.11}
\end{equation}
We obtain the following lemma.

\begin{manuallemma}{B.3} \label{lemmaB.3}
$\displaystyle \int_{\partial M} \mathrm{Tr} \gamma (\partial_u + A)E_1(t,(u,y),(v,hy))|_{\substack{u=v\\y=hy}} dy =0.$
\end{manuallemma}

\begin{proof}
Choose a spectral resolution $\{\varphi_n,\mu_n\}_{n \in \Z \backslash \{0\}}$ of the tangential operator $A$ such that $A\varphi_n=\mu_n\varphi_n,$ $P^\partial \varphi_n=0$ for $n \geq 0;$ and $B=\gamma(\varphi_n)=-\mu_n(\gamma \varphi_n),$  $P^\partial (\gamma \varphi_n)=\gamma \varphi_n$ for $n < 0.$ Then, for $h \in H,$
\begin{align*}
& (\partial_u+A) E_1(t,(u,y),(v,hy)) \\
&= \sum_{n>0} \frac{e^{\mu_n^2t}}{2\sqrt{\pi t}} \left\{ \left(\frac{-2(u-v)}{4t}+\mu_n\right) e^{\frac{-(u-v)^2}{4t}} + \left(\frac{2(u+v)}{4t} -\mu_n\right) e^{\frac{-(u+v)^2}{4t}} \right\} \varphi_n(y) \otimes h^*\varphi^*_n(hy) \\
&+ \sum_{n<0} \frac{e^{\mu_n^2t}}{2\sqrt{\pi t}} \left\{ \left(\frac{-2(u-v)}{4t}-\mu_n\right) e^{\frac{-(u-v)^2}{4t}} - \left(\frac{2(u+v)}{4t} +\mu_n\right) e^{\frac{-(u+v)^2}{4t}} \right\} \varphi_n(y) \otimes h^*\varphi^*_n(hy) \\
& -\sum_{n<0} \left[\mu_n^2 e^{\mu_n(u+v)} \mathrm{erfc}\left(\frac{u+v}{2\sqrt{t}}+ \mu_n\sqrt{t}\right) + \frac{1}{\sqrt{\pi t}}\mu_n e^{-\mu_n(u+v)} e^{-\left(\frac{u+v}{2\sqrt{t}}+ \mu_n\sqrt{t}\right)^2}\right] \varphi_n(y) \otimes h^*\varphi^*_n(hy) \\
&+ \sum_{n<0} \mu_n^2 e^{\mu_n(u+v)} \mathrm{erfc}\left(\frac{u+v}{2\sqrt{t}}+ \mu_n\sqrt{t}\right) \varphi_n(y) \otimes h^*\varphi^*_n(hy).
\end{align*}
When $u=v,$ it follows from the symmetry of spectrum of $A$ that the sum $(\partial_u+A) E_1$ vanishes and thus the integral over $\partial M$ is identically zero.
\end{proof}

Furthermore, one observes that if we split the trace into two separate sums, and rewrite $\gamma AE_1$ in a matrix form under $\Z_2$ grading, the graded terms lie off-diagonally. Thus, we obtain $\displaystyle \int_{\partial M} \mathrm{Tr} \gamma  AE_1(t,(u,y),(v,hy))|_{\substack{u=v\\y=hy}} dy =0.$ From the previous lemma, we have 
\begin{equation}\label{eq:int2}
\displaystyle \int_{\partial M} \mathrm{Tr} \gamma \partial_u E_1(t,(u,y),(v,hy))|_{\substack{u=v\\y=hy}} dy =0. \tag{B.12}
\end{equation}

We need to show there exists a positive constant $C$ such that for any $0 \leq u \leq 1/8,$
\begin{equation}\label{eq:estimate1}
\left| \int_{\partial M} \mathrm{Tr} [\gamma (\partial_u + A) + K_1] E(t,(u,y),(v,hy))|_{\substack{u=v\\y=hy}} dy \right| < C, \quad y \in \partial M, h\in H. \tag{B.13}
\end{equation}
where $E(t,(u,y),(v,z))$ refers to heat kernel of $e(t).$ By the commutativity (resp. anticommutativity) of the involution $\gamma$ with $e_1$ and $K_2$ (resp. $A$ and $K_1$), the integral with the terms $\gamma A E$ and $K_1E$ both vanish, cf. \cite[pg 431]{wojciechowski1999zeta}. By combining all discussion above, it reduces to study 
\begin{equation}
    \int_{\partial M} \mathrm{Tr} \gamma \partial_u   \left(\sum^n_{n=1} \{e_1 \ast_n K_2e_1\}\right)(t,(u,y),(v,hy))|_{\substack{u=v\\y=hy}} dy, \tag{B.14}
\end{equation}
for which we need to estimate the first term $\displaystyle \int_{\partial M} \mathrm{Tr} \gamma \left(\partial_u e_1\right) \ast K_2e_1(t,(u,y),(v,hy))|_{\substack{u=v\\y=hy}} dy.$
At this stage, the lengthy computation and argument follow almost verbatim as in \cite{wojciechowski1999zeta}, so we only state the relevant result.

\begin{manualtheorem}{B.4} \label{thmB.4} \cite[Compare Cor 2.2]{wojciechowski1999zeta}
Let $p(u)$ be a non-increasing smooth function equal to $1$ for $u\leq 1/8$ and equal to $0$ for $u \geq 1/4.$ Then, there exists  $c>0$ such that 
\begin{align*}
\left| \mathrm{Tr} p(u) \gamma\{\partial_u (e_1) \ast K_2e_1 \}(t)\right| \leq c\sum \mathrm{Tr}(h^*_\lambda)\mathrm{Tr}|K_2|,  \\
\left| \mathrm{Tr} p(u) \{e_1 \ast K_2e_1 \}(t)\right| \leq c \sqrt{t}\sum \mathrm{Tr}(h^*_\lambda)\mathrm{Tr}|K_2|, 
\end{align*}
where $h^*_\lambda$ is the induced linear map on the $\lambda$-eigenspace and the sum is over all nonzero $\lambda$.
\end{manualtheorem}

Finally, following \cite[pg 433]{wojciechowski1999zeta} we have the estimate
\begin{equation}
    \left| \mathrm{Tr} p(u) \{\gamma(\partial_u +A) + K_1\} e(t)\right| \leq c_1 \sum \mathrm{Tr}(h^*_\lambda) \mathrm{Tr}|K_2| e^{c_2 t ||K_2||} \tag{B.15}
\end{equation}
for some positive constants $c_1$ and $c_2,$ which completes the proof of Theorem \ref{thmB.1}. 

%%%%%%%%%%%%%%%%%%%%%%%%%%%%%%%%%%%%%%%%%%%%%%%%%%%%%%%%%

\bibliography{bibliography}{}
\bibliographystyle{amsplain}
\end{document}